\documentclass[12pt,reqno]{amsart}
\usepackage{amsmath} \usepackage{amsfonts} \usepackage{amssymb}

\makeindex

\newcommand{\hc}{\mbox{$\mathbb{C_\infty}$}}

\newcommand{\hDi}{\widehat{D}}
\newcommand{\hf}{\hat{f}}
\newcommand{\hg}{\hat{g}}
\newcommand{\si}{\sigma}
\newcommand{\Si}{\Sigma}
\newcommand{\ta}{\theta}
\newcommand{\ga}{\gamma}
\newcommand{\ph}{\varphi}
\newcommand{\om}{\omega}
\newcommand{\al}{\alpha}
\newcommand{\be}{\beta}
\newcommand{\iy}{\infty}
\newcommand{\bd}{\partial}

\newcommand{\mc}{\mathcal}

\newcommand{\ol}{\overline}

\newcommand{\inte}{\mathrm{Int}}
\newcommand{\sn}{\mathrm{sp}}
\newcommand{\ch}{\mathrm{CH}}
\newcommand{\len}{\mathrm{Len}}
\newcommand{\id}{\mathrm{id}}
\newcommand{\plane}{\mathbb{C}}
\newcommand{\sphere}{\mbox{$\mathbb{C_\infty}$}}
\newcommand{\nat}{\mathbb{N}}
\newcommand{\real}{\mathbb{R}}

\newcommand{\disk}{\mathbb{D}}

\newcommand{\complex}{\plane}

\newcommand{\bbd}{\mathbb{D}}

\newcommand{\ucirc}{\mathbb{T}}
\newcommand{\T}{\mathbb{T}}

\newcommand{\mh}{\mc{H}}
\newcommand{\0}{\emptyset}
\newcommand{\sse}{\subsection}

\newcommand{\orb}{\mathrm{orb\,}}

\newcommand{\lam}{\mathcal{L}}
\newcommand{\Ho}{\mathcal{H}}
\newcommand{\M}{\mathcal{M}}
\newcommand{\A}{\mathcal{A}}
\newcommand{\Lstar}{\lam^*}

\newcommand{\eq}[1]{\lfloor#1\rfloor}
\newcommand{\nin}{\not \in}

\newcommand{\ty}{\mathcal T}
\newcommand{\tp}{\mathcal {TP}}

\newtheorem{thm}{Theorem}[section]
\newtheorem{lem}[thm]{Lemma} 
\newtheorem{cor}[thm]{Corollary}
\newtheorem{prop}[thm]{Proposition}

\theoremstyle{definition}

\newtheorem{dfn}[thm]{Definition}

\newtheorem{rem}[thm]{Remark}

\begin{document}

\date{\today} 
\title[Non-degenerate quadratic laminations]
{Non-degenerate quadratic laminations}

\author[A.~Blokh]{Alexander Blokh}
\email[Alexander Blokh]{ablokh@math.uab.edu}
\thanks{The first author was partially
supported by NSF grant DMS-0456748}

\author[D.~K.~Childers]{Douglas K.~Childers}
\email[Douglas K.~Childers]{childers@math.uab.edu}

\author[J.~C.~Mayer]{John C.~Mayer}
\email[John C.~Mayer]{mayer@math.uab.edu}

\author[L.~Oversteegen]{Lex Oversteegen}
\email[Lex~Oversteegen]{overstee@math.uab.edu}
\thanks{The fourth author was partially  supported
by NSF grant DMS-0405774}

\address[Alexander~Blokh, Douglas K.~Childers, John C.~Mayer and Lex Oversteegen]
{Department of Mathematics\\ University of Alabama at Birmingham\\
Birmingham, AL 35294-1170}



\begin{abstract} We give a combinatorial criterion for a critical
diameter to be compatible with a non-degenerate quadratic
lamination.
\end{abstract}

\maketitle

\section{Introduction}\label{intro}

Laminations were introduced by Thurston \cite{thu85} as a tool for studying
complex polynomials, especially in degree $2$. Let $P:\hc\to\hc$ be a degree
$d$ polynomial with a connected Julia set $J_P$, with $\hc$ being the complex
sphere. Denote by $K_P$ the corresponding filled-in Julia set and by $\bbd$ the
closed unit disk. Let $\ta_d=z^d: \bbd\to \bbd$. There exists a conformal
isomorphism $\Psi:\inte \, \bbd\to \hc\setminus K_P$ with $\Psi\circ \ta=P\circ
\Psi$ \cite{hubbdoua85}. If $J_P$ is locally connected, then $\Psi$ extends to
a continuous function $\ol{\Psi}: \bbd\to \ol{\hc\setminus K_P}$ and $\ol{\Psi}
\circ\, \ta=P\circ \ol{\Psi}$. Identify the circle $\partial \bbd$ with $\ucirc
= \mathbb{R}/\mathbb{Z}$. Let $\si_d=\ta_d|_{\ucirc}$,
$\psi=\ol{\Psi}|_{\ucirc}$. Define an equivalence relation $\sim_P$ on $\ucirc$
by $x \sim_P y$ if and only if $\psi(x)=\psi(y)$. The equivalence $\sim_P$ is
called the \emph{($d$-invariant) lamination (generated by $P$)}. The quotient
space $\ucirc/\sim_P=J_{\sim_P}$ is homeomorphic to $J_P$ and the map
$f_{\sim_P}:J_{\sim_P}\to J_{\sim_P}$ \emph{induced} by $\si_d$ is
topologically conjugate to $P|_{J_P}$.

Kiwi \cite{kiwi97} extended this construction to \emph{all} polynomials $P$
with connected Julia set and no irrational neutral cycles for which he obtained
a $d$-invariant lamination $\sim_P$ on $\ucirc$ such that
$J_{\sim_{P}}=\ucirc/\sim_P$ is a locally connected continuum and $P|_{J_{P}}$
semi-conjugate to the induced map $f_{\sim_P}:J_{\sim_P}\to J_{\sim_P}$ by a
monotone map $m:J_P\to J_{\sim_{P}}$ (by \emph{monotone} we mean a map whose
point preimages are connected). The lamination $\sim_P$ generated by $P$
provides a combinatorial description of the dynamics of $P|_{J_{P}}$. One can
introduce laminations abstractly as equivalence relations on $\ucirc$ with
certain properties similar to those of laminations generated by polynomials; in
the case of such an abstract lamination $\sim$ we call $J_\sim=\ucirc/\sim$ a
\emph{topological Julia set} and denote the map \emph{induced} by $\si_d$ on
$J_\sim$ by $f_\sim$. Given a set $A\subset \ucirc\subset \hc$, denote by
$\ch(A)$ the convex hull of $A$ in $\hc$. For an equivalence $\sim$ on $\ucirc$
its graph $G(\sim)\subset \ucirc\times \ucirc$ is the set of all pairs $\{x,
y\}$ such that $x\sim y$; an equivalence is \emph{closed} if its graph is
closed (then all its classes are closed too). Two closed sets $A, B\subset
\ucirc$ are said to be \emph{unlinked} if $\ch(A) \cap \ch(B) = \0$.

\begin{dfn}\label{DEF Lam}
A closed equivalence relation $\sim$ on $\ucirc$ with nowhere-dense classes is
called a \emph{lamination} if its classes are pairwise unlinked.
\end{dfn}

Abusing the language we call the equivalence relation on $\ucirc$ which
identifies all points the \emph{degenerate lamination}. The classes of
equivalence of $\sim$ will be called \emph{$\sim$-classes} (or simply
\emph{classes}).

For points $x,y \in \ucirc$, we use $[x,y], (x,y), \dots$ to denote the
non-empty closed (open, $\dots$) arc running counterclockwise in $\ucirc$ from
$x$ to $y$ (thus, $(x,x)=\ucirc\setminus \{x\}$). For closed $A\subset \ucirc$,
say that $\si_d|A$ is \emph{consecutive preserving} \cite{kiwi97} if for every
component $(s,t)$ of $\ucirc \setminus A$, the interval $(\si_d(s),\si_d(t))$
is a component of $\ucirc \setminus \si_d(A)$.  If, moreover, $\si_d(A)=A$ for
a closed set $A$, we say $A$ is {\em rotational}.

\begin{dfn}\label{DEF InvLam}
The lamination $\sim$ is \emph{$d$-invariant} if for every class $C$ the
set $\si_d(C)$ is a class. A class $C$ is \emph{critical} if
$\si_d|_C$ is not injective.
\end{dfn}

\begin{rem}\label{notconpre} It follows that the preimage of a class is a union of classes.
From now on we consider the \emph{quadratic}
case $d=2$; by $\si$ we mean $\si_2$ and by \emph{invariant} we mean
$2$-invariant. In the literature it is also required that for an
invariant lamination and each class $C$, the map $\si|_C$ is
consecutive preserving. We show in Lemma~\ref{REM *} that this
assumption is redundant in case $d=2$.
\end{rem}

Thurston's original approach was different from that described
above. He did not consider equivalences on $\ucirc$. Rather, he
considered closed families of chords in $\ucirc$ having specific
properties. Following \cite{thu85}, call a chord $\ol{ab}$ joining
two points $a, b\in \ucirc$ a \emph{leaf} (we allow for the
possibility that $a=b$ in which case the leaf is degenerate).

\begin{dfn}\label{geolam} A \emph{geometric lamination} \cite{thu85}
$\lam$ is a compact set of leaves such that any two distinct leaves from $\lam$
meet at most in an endpoint of both of them. A geometric lamination $\lam$ is
said to be \emph{invariant} if for each $\ell=\ol{cd}\in\lam$,
$\ol{\si(c)\si(d)}$ is a leaf in $\lam$ and there exist two disjoint leaves
$\ell'=\ol{c'd'}$ and $\ell''=\ol{c''d''}$ in $\lam$ such that
$\si(c')=\si(c'')=c$ and $\si(d')=\si(d'')=d$.
\end{dfn}

Geometric laminations serve as a tool for studying non-locally connected
Julia sets \cite{OB}. An advantage of considering them is that a geometric
lamination can be constructed if only one (but
an appropriately chosen one) of its leaves is known. Given a geometric lamination
$\lam$, we denote by $\lam^*$ the union the circle $\ucirc$ and of all the leaves of $\lam$. By a
\emph{gap} $G$ of $\lam$ we mean the closure of a component of $\disk\setminus\lam^*$.

The construction of a geometric lamination from a single leaf is due to
Thurston and is described in the next section. An important case is when a
\emph{critical leaf (diameter)} is given and the entire geometric lamination to
which it belongs needs to be recovered (given a point $\ta\in \ucirc$ we set
$\ta'=\ta+1/2$ and denote the corresponding critical diameter $\ol{\ta\ta'}$
by $\ell_\ta$). In a lot of cases this recovery can be done
completely. Still, the problem is to relate \emph{geometric} laminations (or,
alternatively, critical diameters which determine them) and their
\emph{equivalence} counterparts. More precisely, a lamination $\sim$ is said to
be \emph{compatible with a critical diameter $\ell_\ta$} if $\ta\sim \ta'$. We solve
the following problem in the paper.

\smallskip

\noindent\textbf{Main Problem.} \emph{Given a critical diameter $\ell_\ta$,
does there exist a non-degenerate lamination compatible with it?}

\smallskip

In one direction the connection between laminations and geometric
laminations is not hard.

\begin{dfn}\label{assoclam} Let $\sim$ be a lamination.  The \emph{
geometric lamination $\lam_\sim$} is formed as follows:
take for each $\sim$-class $A$ its convex hull $\ch(A)$. Take all chords,
including possibly degenerate ones, in the boundary of $\ch(A)$ to be leaves of
$\lam_\sim$. The family of so-constructed leaves, degenerate and otherwise,
over all $\sim$-classes is $\lam_\sim$.
\end{dfn}

By our assumption the boundary of each gap is the union of  chords and points; the
non-degenerate chords become leaves while points become degenerate
leaves. For example, if the class $A$ of $\sim$ is a Cantor set, then all points of $A$ which are
not the endpoints of complementary arcs to $A$ will be degenerate leaves. Also,
classes which consist of two points become leaves, and classes consisting of
one point, become degenerate leaves. In this way we construct a geometric
lamination (so that its leaves can only meet on the unit circle) such that any
two points connected with a leaf are equivalent in the sense of $\sim$.   It is
left to the reader to complete the proof of the following proposition.

\begin{prop}\label{lamsim} The collection of leaves $\lam_\sim$ associated to the
lamination $\sim$ is a geometric lamination. \end{prop}

To solve the Main Problem we proceed in the opposite direction. Given a
critical leaf, we construct the corresponding geometric lamination as in
\cite{thu85} to which we then associate a lamination whose non-degeneracy we
study. The paper is organized as follows. First we establish useful properties
of invariant laminations in Section~\ref{invlam}. In Section~\ref{geola} we
discuss properties of geometric laminations as well as ways of constructing
them. In Section~\ref{nonper} we show that if the point $\si(\ta)$ is not
periodic then there exists a lamination compatible with $\ell_\ta$.

The main case when $\si(\ta)$ is periodic is considered in Section~\ref{per}
and Section~\ref{relam}. In Section~\ref{per}, we develop the notion of
renormalization of a ``quadratic" map on a dendrite, analogous to the notion of
renormalization of a unimodal map on an interval.   We apply this to
laminations in the subsequent section. In Section~\ref{relam} we describe two
basic cases, {\em basic rotational} and {\em basic non-rotational}, and use
them to develop an algorithmic verification test of whether a non-degenerate
lamination compatible with $\ell_\ta$ exists with three possible outcomes: (1)
basic rotational, hence degenerate; (2) basic non-rotational, hence
non-degenerate, or (3) inconclusive.  In the inconclusive case, we develop a
version of renormalization of invariant laminations, in effect reducing the
period of $\si(\ta)$, and apply the test again. Ultimately, for $\si(\ta)$
periodic, the algorithm terminates in case (1) or (2).  We then use the results
of Section~\ref{per} to conclude that the original lamination is, respectively,
degenerate or non-degenerate.

We extract from this verification algorithm a combinatorial description of a
{\em block structure} for a periodic orbit which fully describes those periodic
angles $\ta$ such that $\ell_\ta$ is compatible with a non-degenerate
lamination.

\smallskip

\noindent\textbf{Acknowledgements} This paper is related to discussions in the
Lamination Seminar at UAB in 2004-2006. We want to thank participants of this
seminar for many discussions which have illuminated our understanding of
laminations.

\smallskip

\section{Invariant Laminations}\label{invlam}

\sse{Fundamental Properties} In this section we study fundamental
properties of laminations which will be used in subsequent proofs.  For a set
$A\subset \ucirc$ let $A'$ be the image of $A$ under rotation by $1/2$.

\begin{lem}\label{REM *} Let $\sim$ be an invariant lamination. For a class
$C$ either {\rm (a)} $C$ is critical, $C'=C$ and $\si^{-1} \circ \si(C) = C$,
or {\rm (b)} $C$ is not critical, $C'$ is another class with $C'\cap C=\0$, and
$\si^{-1} \circ \si(C) = C \cup C'$. If there exists a critical class
$C$ then $C$ is unique and $\ch(C)$ contains a diameter. Also, if $A$ is
a class then $\si|_A$ is consecutive preserving.
\end{lem}

\begin{proof} By \cite{BS} a subset
$A\subset\ucirc$ is mapped consecutive-preserving under $\sigma$
provided $A$ is contained in a closed semicircle. In that case,
$\sigma|_A$ is one-to-one except possibly at the endpoints of the
semicircle. To prove {\rm (a)}, suppose $C$ is a critical class.
Then there are points $c, c'\in C$. Suppose that there is yet
another point $b\in C$. Then $\si(C) \ne \si(b)\in \si(C)$ where by
invariance $\si(C)$ is a class. Denote by $B$ the class containing
$b'$. Then the class $\si(B)$ is non-disjoint from the class
$\si(C)$, hence $\si(C)=\si(B)$. Since $\si(c)\in \si(C)$ then $B$
must contain either $c$ or $c'$, thus $B$ must coincide with $C$ and
$b'\in C$. Hence $C=C'$ and $\si^{-1} \circ \si(C) = C$. Now, if
there were another critical class, it would not be unlinked with $C$
because, by the above, it would have to contain a diameter, and
diameters meet.  It follows that a critical class $C$ is unique. To
prove {\rm (b)}, suppose $C$ is not critical. Then by definition,
$C\cap C'=\0$, and $C'$ is a class. So, $C$ is unlinked with $C'$,
and since $C$ and $C'$ are both closed, they are contained in
opposite open semicircles. Clearly, $\si(C)=\si(C')$ and $\si^{-1}
\circ \si(C) = C \cup C'$. It follows from this that $\si$ is
consecutive-preserving on a class.
\end{proof}

If $\sim$ is an invariant non-degenerate lamination then, by the expanding
properties of $\si$, no $\sim$-class has interior in $\ucirc$. The quotient
space $\ucirc/\sim$ can be embedded in $\sphere$ as a locally connected
continuum $J_{\sim}$. Indeed, we can start by considering the sphere with the
unit disk. Then we can consider the equivalence on the sphere which extends
$\sim$ by identifying points of every convex hull of a $\sim$-class and not
identifying any points outside the unit disk. Denote the corresponding
identifying factor map of the sphere onto the sphere by $\pi$. The radial rays
which connect points of the unit circle to infinity under $\pi$ are mapped into
the so-called \emph{topological external rays}. One can consider the map
$z\mapsto z^n$ outside the open unit disk and then
transport this map by $\pi$ onto the sphere $\pi(S^2)$. This induces the map
$\hf_\sim$ of the $\pi$-image of complement of the open unit disk onto itself
which extends the induced map $f_\sim$. By
Kiwi \cite{kiwi97} the map $\hf_\sim$ extends to a
branched covering map $\hf_\sim: \sphere \to \sphere$ of degree $2$ \emph{of
the entire sphere} whose dynamics resembles that of a polynomial
(\cite{kiwi97} applies to \textbf{all} degrees but we state them for degree 2).
The map $\hf_\sim$ is called a \emph{topological extension of} $f_\sim$.

No we need the following two concepts.

\begin{dfn} Let $X$ be any space and $f:X\to X$ a map. A set $K\subset X$ is
said to be \emph{wandering} iff for every $m \neq n$, $f^n(K) \cap f^m(K) =
\0$.  The map $f$ has an \emph{identity return} iff there exist a continuum
$K\subset X$ (not a point) and an integer $n>0$ with $f^n|_K = \id|_K$.
\end{dfn}

It is proven in \cite{BL1}, that $f_\sim$ has no wandering continua. Let
$J\subset\sphere$ be a compact set and let $g:\sphere\to\sphere$ be a branched
covering map such that $g(J)=J=g^{-1}(J)$. A complementary component $U$ of $J$
is called a \emph{domain}. Some properties of the map $f_\sim$ are listed in
Proposition~\ref{TH NoWandDom or IdRet}. A {\em critical point} of a map is a
point for which there is no neighborhood on which the map is one-to-one.

\begin{prop}\label{TH NoWandDom or IdRet}
If $\sim$ is an invariant lamination then:

\begin{enumerate}

\item the induced map $f_\sim$ has no wandering continua and the
extension $\hf_\sim$ has no wandering domains;

\item  for any  $x\in J_{\sim}$ such that $f_\sim^n(x)=x$, there
exists a neighborhood $U\subset J_{\sim}$ of $x$ such that any $y\in
U\setminus x$ eventually exits $U$ under iterations of $f_\sim^n$
(in particular, $f_\sim$ has no identity return).

\end{enumerate}

\end{prop}

\begin{proof} (1) By \cite{BL1} $f_\sim$ has no wandering continua.
This easily implies that $\hf_\sim$ has no wandering domains.
Indeed, first observe that all points of $J_\sim$ are accessible
from the \emph{basin of infinity}; such sets are said to be
\emph{unshielded} \cite{BO1}. Let $U$ be a wandering domain of
$J_\sim$. Then since $J_\sim$ is locally connected and unshielded,
$\partial U$ is homeomorphic to the unit circle $\T$.

Since $\bd U$ is a continuum, then it is non-wandering and for some
integers $n\ge 0, m>0$ we have $A=\hf_\sim^n(\bd U)\cap
\hf_\sim^{n+m}(\bd U)\ne \0$. Moreover, $A=\{a\}$ is a singleton,
for otherwise there will be points of $J_\sim$ shielded from
infinity. We may assume that $n=0$ and consider $\hg=\hf_\sim^m$
instead of $\hf_\sim$, so that $\{a\}=\bd U\cap \hg (\bd U)$. The
trajectory of the set $\bd U$ is a sequence of Jordan curves,
enclosing pairwise disjoint Jordan disks, consecutively attached to
each other at a sequence of points $\{\hg^i(a)\}$ such that for
fixed $i\ge 1$, $\hg^{i-1}(\bd U)$ meets $\hg^{i}(\bd U)$ at
$\hg^i(a)$, and, similarly, $\hg^{i}(\bd U)$meets $\hg^{i+1}(\bd U)$
at $\hg^{i+1}(a)$.  Since $J_\sim$ is unshielded, these are the only
points at which forward images of $\bd U$ can meet $\hg^{i}(\bd U)$.

We will consider two possibilities. Suppose first that $\hg(a)\ne a$. Then
$\hg^i(\bd U)\cap \hg^j(\bd(U)=\0$ when $j-i\ge 2$. Hence we may assume, without loss of generality,
that $\hg^i(\ol{U})$ does not contain a critical point of all $i$ and, hence, $\hg$ is a homeomorphism
on $\hg^i(\ol{U})$ for each $i\ge 0$. Let  $K\subset\hg(\bd U)$ be a closed arc disjoint from $\{a,\hg(a)\}$. Since
$\hg^i$ is a homeomorphism on $\hg(\bd U)$, then $\hg^i(K)$ is disjoint from
$\{\hg^i(a),\hg^{i+1}(a)\}$.  By the previous paragraph, no forward images of
$\bd U$ can meet $\hg^i(K)$.  Hence $K$ is wandering, a contradiction. If $g(a)=a$ then we may also assume,
 as above,  that $\hg^i(\ol{U})$ does not contain a critical point of $g$ for $i$ sufficiently large.
 Hence $\hg|_{\hg^i(\ol{U})}$ is a homeomorphism and any continuum $K\subset \hg^i(\ol{U})\setminus\{a\}$
 would be a wandering continuum, a contradiction.

(2) The claim is proven in \cite[Lemma 3.8]{BO1}.
\end{proof}

\sse{Kneadings}\label{kneadings} Let us discuss some results and notions
introduced in \cite[Section 4.3]{kiwi05} by Kiwi. We are interested in the case
when $d=2$; in this case Kiwi calls a pair of points $(\ta, \ta')$ a
\emph{critical portrait} (we call $\ol{\ta \ta'}$ a \emph{critical diameter})
for which he introduces \emph{aperiodic kneadings}. The critical diameter
$\ol{\ta \ta'}$ divides $\ol{\bbd}$ into two components $B_1, B_2$ whose
intersections with $\T$ are two open semicircles with endpoints $\ta, \ta'$.
Given $t\in \T$, its \emph{itinerary i(t)} is the sequence $I_0, I_1, \dots$ of
sets $B_1, B_2, \{\ta, \ta'\}$ with $\si^n(t)\in I_n (n\ge 0)$. A critical
diameter $\ol{\ta \ta'}$ such that $i(\si(t))$ is not periodic is said to have
a \emph{aperiodic kneading} (our definition is equivalent to that given by Kiwi
in \cite{kiwi05}). Call a lamination $\sim$ \emph{compatible} with a critical
portrait $(\ta, \ta')$ if $\ta\sim \ta'$. The results of \cite{kiwi05} in the
quadratic case imply that a critical diameter with aperiodic kneading has a
compatible non-degenerate lamination. This leaves open the question of the
existence of a compatible lamination when a critical portrait $(\ta, \ta')$ has
a periodic kneading, in particular when $\ta$ or $\ta'$ is periodic. Solving
this problem is our main result. The case when $\ta$ and $\ta'$ are not
periodic, but have periodic kneading, does not follow directly from Kiwi's
results and is addressed in Section~\ref{nonper}.

\section{Geometric laminations}\label{geola}

We follow Thurston \cite{thu85} but address mostly the case $d=2$. Let us
give a geometric interpretation to a lamination $\sim$. Namely, given
any $\sim$-class $g$ let us consider its convex hull $\ch(g)$. If $g$ is a
point then $\bd(\ch(g))=g$ is a point; if $g$ consists of two points then
$\bd(\ch(g))=\ch(g)$ is a chord of $\disk$. Finally, if $g$ consists of more
than two points then the boundary of $\ch(g)$ consists of chords of $\disk$ and
points of $\ucirc$. The union of all the boundaries of all $\sim$-classes is
denoted by $\lam(\sim)$ and is called the \emph{geometric lamination of
$\sim$}. The way the chords from $\lam(\sim)$ (Thurston calls them
\emph{leaves}) map onto each other is quite specific and can be formalized
which was done by Thurston in \cite{thu85} where these properties of leaves are
postulated and taken as the definition of the corresponding unions of leaves
called \emph{geometric laminations}.

The aim of our paper is to do the opposite, i.e. given a geometric
lamination to recover a lamination so that any two endpoints of a
leaf are equivalent. In the quadratic case we associate to a given
critical diameter the associated geometric lamination and then study
if there is any lamination corresponding to it in the above sense.
Also, we describe an algorithm which allows one to associate a
lamination to a given geometric lamination. First we would like to
remove one particular case from consideration. Namely, the
\emph{vertical lamination $V$} is defined by $xVy$ if and only if
$x=\pm \,y$. The corresponding geometric lamination consists of all
vertical chords of $\ucirc$. Observe that the vertical lamination
$V$ is compatible with the critical diameter $(1/4, 3/4)$; this
solves the main problem of the existence of a compatible lamination
for the critical diameter $(1/4, 3/4)$. Thus from now on we
\emph{always} assume that the \emph{critical diameter is not
vertical} and consider \emph{only geometric laminations which are
not the vertical lamination}.

\sse{Fundamental Construction} \label{LamConstruction} Suppose that
$\ell_\ta$ is a critical diameter. Put $\ell_\ta=\ol{\ta\ta'},
E_0=\{\ta, \ta'\}$ and let $\lam^\ta_0=\{ \ell_\ta\}$. Then
$\si^{-1}(E_0)=E_1$ is a set of 4 points disjoint from $E_0$. If
$0\nin \{\ta, \ta'\}$ we pair up these four points into two sets of
two points $\{a_1,b_1\}$ and $\{a_2,b_2\}$ so that the union of
$\lam^\ta_0$ and the two leaves $\ol{a_i b_i}$ is a geometric
lamination $\lam^\ta_1$. On the other hand, suppose that $a_0=0,
b_0=1/2$. Then $E_0\cup E_1=\{0, 1/2, 1/4, 3/4\}$ which includes two
points we have already visited. In this case we connect $1/4$ to
both $0$ and $1/2$ by means of two leaves, and we connect $3/4$ to
both $0$ and $1/2$ by means of two leaves also. In this way a finite
geometric lamination $\lam^\ta_1$ is created. Then we make another
pull back and create $\lam^\ta_2$, etc.

Assume $\lam^\ta_n$ has been constructed. To create the next
geometric lamination $\lam^\ta_{n+1}$ we add to $\lam^\ta_n$
\emph{all possible preimage leaves of leaves from $\lam^\ta_n$ which
are unlinked with the leaves of $\lam^\ta_n$} (strictly speaking, we
cannot talk of preimage leaves since the map is not defined inside
the unit disk - we talk about them meaning that their endpoints map
to the endpoints of their image leaves). Some preimages of leaves
from $\lam^\ta_n$ already belong to $\lam^\ta_n$. However, there
will be other, new preimages as well. If $\ta, \ta'$ are not
periodic then it is easy to see that all preimage leaves constructed
as above are pairwise disjoint. A bit more complicated picture holds
if $\ta$ or $\ta'$ is periodic (as in the case where $\ta=0$ above);
then the leaves of the geometric lamination $\lam^\ta_{n+1}$ will
not be pairwise disjoint although they can only meet at their
endpoints in $\ucirc$.

Let us show by induction that $\lam^\ta_{n+1}$ is a lamination. Clearly,
$\lam^\ta_1$ is a lamination. Now, by the construction new preimage leaves
cannot cross the old ones inside $\disk$. Suppose, by way of contradiction,
that two new preimage leaves cross each other inside $\disk$. Then their images
must cross inside $\disk$ too, a contradiction. This is the major principle
upon which Thurston's construction is based.

A leaf $\ell$ belongs to $\lam^\ta_m$ iff $\si^i(\ell)=\ell_\ta$ for some $i\le m$
and $\ell,$ $\si(\ell),$ $\dots,$ $\si^{i-1}(\ell)$ are
unlinked with $\ell_\ta$. Indeed, if $m=1$ it follows from the
construction. Let the claim hold for $m=k$ and prove it for $m=k+1$.
If $\ell\in \lam^\ta_{k+1}$ then by the construction it has to satisfy the
listed above conditions. Now, suppose that for some $i\le m$ we have that
$\si^i(\ell)=\ell_\ta$ and $\ell, \si(\ell), \dots, \si^{i-1}(\ell)$ are
unlinked with $\ell_\ta$. If $i\le k$ then $\ell\in \lam^\ta_m$ by induction.
If $i=k+1$ then $\si(\ell)\in \lam^\ta_m$ by induction. Let us show that $\ell$
is unlinked with all leaves in $\lam^\ta_k$. Indeed, by the assumptions it is
unlinked with $\ell_\ta$; if it crosses another leaf $\ell'$ of $\lam^\ta_k$
inside $\disk$ then its image will cross the image of $\ell'$, a contradiction
with $\si(\ell)\in \lam^\ta_k$. This proves that $\ell\in \lam^\ta_m$ iff
$\si^i(\ell)=\ell_\ta$ and $\ell, \si(\ell), \dots, \si^{i-1}(\ell)$ are
unlinked with $\ell_\ta$ for some $i\le m$. We will use this claim to check if
a preimage leaf of $\ell_\ta$ belongs to a finite pullback lamination.

First suppose $\ta,\ta'$ are non-periodic. On the first step
$\ol{\ta\ta'}$ divides the unit disk $\disk$ into two half-disks
with $\si$-images of their semicircles being $\ucirc$. Then two
first preimages of $\ol{\ta\ta'}$ are added, and since $\ta,\ta'$
are non-periodic, the new leaves are disjoint from $\ol{\ta\ta'}$
and hence from each other. The collection of thus created leaves
divides $\disk$ into 3 subsets whose boundaries intersect $\ucirc$
over finite collections of arcs with first images being the
semicircles from the previous step and the second images being
$\ucirc$. Inductively, the same picture holds on every step.

More precisely, on the step $n$ we have $2^n-1$ pairwise disjoint
leaves which partition $\disk$ into $2^n$ sets. Each element $A$ of
the partition has the boundaries consisting of the union $S_A$ of
several arcs of $\ucirc$ and equally many leaves. The $\si$-image of
$S_A$ is the set $S_B$ for the appropriate element $B$ of the
partition of generation $n-1$. Moreover, $B$ is divided by the
appropriate leaf $\ell$ of generation $n$ into two partition
elements $B', B''$ of generation $n$. Then the preimage of $\ell$
inside $A$ is the new leaf of generation $n+1$ which should be added
now. Its endpoints do no coincide with endpoints of leaves of
previous generations because $\ta,\ta'$ are not periodic. Thus, by
induction for each $n$ there exists a finite geometric lamination
$\lam^\ta_n$, with \emph{pairwise disjoint} leaves, such that if
$E_n$ is the set of endpoints of leaves of $\lam^\ta_n\setminus
\lam^\ta_{n-1}$, then $\sigma^{-1}(E_n)=E_{n+1}$ is the set of
endpoints of leaves from $\lam^\ta_{n+1}\setminus \lam^\ta_n$.

If one of $\ta,\ta'$ is periodic with least period $n>1$, then the
leaves of $\lam^\ta_{n+1}\setminus \lam^\ta_n$ will not be pairwise
disjoint. However, the leaves are unlinked and meet in at most an
endpoint of each, as in the case of pulling back the critical leaf
$\ol{0\frac12}$.  This happens exactly when an endpoint of a
preimage leaf of $\lam^\ta_n$ pulls back to the \emph{critical
value} $\si(\ta)=\si(\ta')$.

The set $\cup_n \lam_n$ is a countable union of pairwise disjoint
leaves which are the preimages of the leaf $\ell_\ta$.  We call
$\lam^\ta=\cup_n \lam_n$ the {\em pre-lamination} generated by
$\ell_\ta$, and we call the leaves of $\lam^\ta$ {\em precritical
leaves}, and we set $\lam=\lam^\ta_{\iy}=\ol{\cup_n \lam_n}$. Hence,
leaves other than precritical leaves must be limits of sequences of
precritical leaves, from one or both sides.  We will call these {\em
limit leaves} of $\lam$. (Of course, precritical leaves could also
be limit leaves.)

Thus $\lam^\ta_{\iy}$ is the collection of all leaves from
${\bigcup_{n=0}^\iy \lam^\ta_n}$ and their limit leaves (the latter
may include degenerate limit leaves, i.e. points of $\T$). Clearly,
$\lam^\ta_{\iy}$ is a closed family of leaves. Theorem~\ref{thu}
below concerns the family $\lam^\ta_{\iy}$ and was proven in
\cite[Proposition II.4.5]{thu85}.  The proof follows from the above
construction, the fact that $\lam^\ta_{\iy}$ is closed, and from the
fact that closure preserves leaves being unlinked.

\begin{thm}\label{thu} The family $\lam^\ta_{\iy}$ is an invariant geometric
lamination.
\end{thm}

Given a geometric lamination $\lam$ we set $\Lstar=\bigcup
\lam\cup\ucirc$, the union of all leaves of $\lam$ and all points of
$\ucirc$. A \emph{gap} $G$ of a geometric lamination $\lam$ is the
closure of a bounded component of the complement of $\Lstar$.  The
boundary of a gap $G$ consists of leaves and points of $\ucirc$,
possibly infinitely many of each.  It is useful to distinguish gaps
whose boundary contains finitely many leaves as {\em finite gaps}
(really, inscribed polygons), and call others {\em infinite gaps}.
Proposition~\ref{gapdens} is proven in \cite{thu85,where?}.

\begin{prop}\label{gapdens} Gaps are dense in any quadratic
invariant geometric lamination $\lam$ provided it is not the vertical lamination.
\end{prop}

It is easy to see that if $\lam$ is invariant then the set $\Lstar$
is a continuum in $\disk$ containing the unit circle, and the
density of gaps simply means that the open set $\disk\setminus
\Lstar$ is dense in $\disk$. In fact, $\Lstar$ is the closure of
$\bigcup_{n=0}^\iy \lam^\ta_n$ where the latter is understood as a
set of points from all the corresponding leaves, not as the
collection of leaves. Since we do not consider the vertical
lamination, from now on we consider \textbf{only geometric
laminations with dense gaps}.

In \cite{thu85} Thurston shows that all gaps in a quadratic invariant geometric lamination that do not collapse to a leaf under iteration are pre-periodic.   We shall not need this fact, but we will need a simpler fact about periodic gaps in the lamination
$\lam^\ta_{\iy}$ constructed above. It is convenient to define the {\em length}
$\len(\ell)$ of a leaf $\ell$ to be the length of the shorter subarc of $\ucirc$ that it subtends.

\begin{prop}\label{pergap}
Let $G$ be a gap of $\lam^\ta_{\iy}$ and suppose there is a least
$n$ such that $\si^n (G)=G$. Let $\ell$ be any leaf in $\bd G$. Then
$\ell$ is either preperiodic or precritical.
\end{prop}

\begin{proof}
By way of contradiction, suppose that $\ell$ is a leaf of $\bd G$
which is neither preperiodic nor precritical.  Consider the iterates
of $\si^n$ on $\ell$.  The sequence $\{\si^{ni}(\ell)\}_{i=0}^\iy$
is an infinite set of non-degenerate leaves in $\bd G$, hence
$\{\len(\si^{ni}(\ell))\}_{i=0}^\iy$ forms a null sequence in
length.  But $\si^n$ is a locally expanding map. Hence, there is a
$\delta>0$ such that $\len(\ell)<\delta\implies
\len(\si^n(\ell))>\len(\ell)$.  There are at most finitely many
leaves of length $\ge\delta$.  Choose $N\in\nat$ such that for all
$i\ge N$, $\len(\si^{ni}(\ell))<\delta$.  But then
$\len(\si^{n(i+1)}(\ell))>\len(\si^{ni}(\ell))$, which contradicts
that the sequence $\{\len(\si^{ni}(\ell))\}_{i=0}^\iy$ is null.
\end{proof}

In the rest of this section we work mainly with geometric
laminations, not necessarily invariant. We need some notions dealing
with equivalences. Equivalences $\sim, \approx$ can be compared in
the sense of their graphs $Gr(\sim), Gr(\approx)$: we say that
$\sim$ is \emph{finer} than $\approx$ if $Gr(\sim)\subset
Gr(\approx)$. Equivalently, $\sim$ is finer than $\approx$ if
$\sim$-classes are subsets of $\approx$-classes. Yet another useful
way to define this is that $\sim$ is finer than $\approx$ if $x\sim
y$ always implies $x\approx y$. Given two closed equivalence
relations $R, Q$ one can define their intersection $P=Q\cap R$ as
the equivalence whose classes are intersections of equivalence
classes of $R$ and $Q$. Clearly, $P$ is a well-defined closed
equivalence relation too. Moreover, if $R, Q$ are laminations then
$P$ is a lamination too. The same can be done not only for two but
for any family of closed equivalence relations (laminations).

Now, suppose that $\lam$ is a geometric lamination. Then a closed equivalence
relation $R$ on $\ucirc$ is said to be \emph{compatible} with $\lam$ if for any
two points $x, y\in \ucirc$ the fact that $\ol{xy}\in\lam$ implies $xRy$.
There exists a finest closed equivalence relation $R$ on $\ucirc$ compatible with
$\lam$. Indeed, observe that the degenerate lamination is compatible with
$\lam$. Define the lamination $R_\lam$ as the intersection of all laminations
compatible with $\lam$ (in other words, declare two points $x, y\in \ucirc$
equivalent, denoted $xR_\lam y$, if they are equivalent in the sense of all the
laminations compatible with $\lam$). Then obviously $R_\lam$ is the finest
lamination compatible with $\lam$.

Given a geometric lamination $\lam$, we study the quotient space
$J_\lam=\ucirc/R_\lam$; we want to determine when $J_\lam$ is not a
point. Since invariant geometric laminations are often obtained by
an infinite process (like the construction of $\lam^\ta_{\infty}$)
it is in general difficult to decide which points are equivalent
and, in particular, when $J_\lam$ is non-degenerate. For this reason
we define a specific lamination $\sim_\lam$ in a different way, and
show that $R_\lam$ is equal to $\sim_\lam$. Lemma~\ref{contain}
studies continua inside $\Lstar$.

\begin{lem}\label{contain} If $K\subset\Lstar$ is a
continuum then the following claims hold.

\begin{enumerate}

\item If $\ell\in\lam$ is a leaf with $K\cap \ell\ne \0$ and $K$ does not contain
an endpoint of $\ell$ then $K\subset \ell$. In particular,
if $K\cap \ucirc\ne\0$ then $K$ contains an endpoint of $\ell$ and
if $K$ meets two distinct leaves then it meets $\ucirc$.

\item If $G$ is a gap and $x, y\in \ucirc\cap G\cap K$ then either
$(x, y)\cap G\subset K$ or $(y, x)\cap G\subset K$.

\end{enumerate}

\end{lem}

\begin{proof} (1) Given a leaf $\ell$ with endpoints $a, b$, choose small disks $U, V$
centered at $a, b$. Set $W=\disk\setminus (U\cup V)$. Since gaps are dense,
arbitrarily close to $\ell\cap W$ from either side there are ``in-gap'' curves
$Q, T$ connecting points of $\ol{U}\cap \disk$ with points of $\ol{V}\cap
\disk$ and disjoint from $\Lstar$. Hence $\ell\cap W$ is a component of
$\Lstar\cap W$. If a continuum $K\subset \Lstar$ is non-disjoint from $\ell$
and does not contain an endpoint of $\ell$ then $U, V$ can be chosen so small
that $K\subset W$ and so by the above $K\subset \ell$.


(2) Suppose that $u\in (x, y)\cap G\setminus K$ and $v\in (y, x)\cap G\setminus
K$. Connect points $u$ and $v$ with an arc $T$ inside $G$. Then $T$ separates
$x$ from $y$ in $\disk$ and is disjoint from $K$, a contradiction.
\end{proof}

We are ready to give a constructive definition of the lamination which, as we
prove later, coincides with $R_\lam$.

\begin{dfn}\label{df-lageo} Call a continuum $K\subset\Lstar$ which meets
$\ucirc$ in a countable set an \emph{$\om$-continuum.} Given a geometric
lamination $\lam$, let $\sim_\lam$ be the equivalence relation in $\ucirc$
\emph{induced by} $\lam$ as follows: $x\sim_\lam y$ iff there exists an $\om$-continuum $K
\subset\Lstar$ containing $x$ and $y$.
\end{dfn}

Clearly, $\sim_\lam$ above is an equivalence relation which is
compatible with $\lam$. For $x\in \ucirc$ we denote by $\eq{x}$ the
$\sim_\lam$-class of $x$.

\begin{thm}\label{omega closed} If $\lam$ is a geometric lamination then
$\sim_\lam$ is a lamination.  Moreover, $\sim_\lam=R_\lam$, the finest
equivalence relation compatible with $\lam$.
\end{thm}

\begin{proof} We show first that equivalence classes are closed.  We may
assume that there exists a sequence $\{x_i\}$ in $\eq{x_1}$ such that
$x_1<x_2<\ldots<x_\infty$ with $\lim x_i=x_\infty$ and
$\ol{x_1x_{\infty}}\not\in\lam$ (otherwise trivially $x_\infty\in \eq{x_1}$),
where $<$ denotes the induced circular order on $\ucirc=\real/\mathbb{Z}$. We
show that $x_\infty\in\eq{x_1}$. Since $x_i\in \eq{x_1}$ for every $i$ then there exists an $\om$-continuum
$K_i$ containing both $x_1$ and $x_i$. Also, let $\lam_{1,\infty}$ be the
collection of all leaves $\ell=\overline{pq}\in\lam$ with $p$ and $q$ in
distinct components of $\ucirc\setminus\{x_1,x_\infty\}$. Because $\lam$ is
unlinked, $\lam_{1,\infty}$ has a linear order defined by $\ell<\ell'$ provided
that $\ell$ separates the (open) disk between $\ell'$ and $x_1$ (the leaves
$\ell$ and $\ell'$ may meet on the unit circle in which case we can only talk
about separation in the open unit disk).

First assume that  $\lam_{1,\iy}=\0$. Then $x_1, x_\iy$ belong to the same gap $G$. By
Lemma~\ref{contain}, $K_i$ contains either $[x_1, x_i]\cap G$ or $[x_i, x_1]\cap
G$ for any $i$. If $[x_i, x_1]\cap G\subset K_i$ then $x_\iy\in K_i$ and hence
$x_1\sim_\lam x_\iy$ as desired. If $[x_1, x_i]\cap G\subset K_i$ for any $i$
then $[x_1, x_\iy]\cap G$ is countable and the part $Q$ of $\bd G$ which
extends, in the positive direction, from $x_1$ to $x_\iy$, is an
$\om$-continuum containing $x_1$ and $x_\iy$. Hence again $x_1\sim_\lam x_\iy$
as desired.

Next, assume that $\lam_{1,\iy}\ne \0$. Let $M=\sup\{\lam_{1, \iy}\}=\ol{pq}$
($p$ and $q$ may coincide, in which case $M=\{x_\iy\}$). Since $\Lstar$ is
closed, $M\subset\Lstar$. Let us show that $\{p, q\}\subset \eq{x_1}$. Indeed,
choose a sequence (or a finite set) of leaves
$\ell_0<\ell_1<\dots$, $\lim \ell_i=M$ and points $x_{n(i)}, i=1, 2, \dots$
such that the component $L_i$ of $\disk\setminus[\ell_{i-1}\cup\ell_i]$, whose
boundary contains $\ell_i$ and $\ell_{i+1}$, separates $x_1$ and $x_{n(i)}$.
Then $(\ol{L_i}\cap K_{n(i)})\cup \ell_{i-1}\cup \ell_i=R_i$ is an
$\om$-continuum which extends from $\ell_{i-1}$ to $\ell_i$. Also, let $R_0$ be
the closure of the intersection of $K_{n(1)}$ with the component of
$\disk\setminus \ell_0$ containing $x_1$ union $\ell_0$. Then $R=M\cup
(\cup_{i=1}^\iy R_i)$ is an $\om$-continuum, and hence $p, q\in \eq{x_1}$. If
$x_\iy\in M$ then we are done. If $x_\iy\nin M$ then it follows from the
definition of $\lam_{1, \iy}$ that there exists a gap $G$ such that $\bd G$
contains $M$ and $x_\iy$. Let $p\in (x_1, x_\iy)$. Then the argument from the
previous paragraph applies to $p$ (playing the role of $x_1$) and $x_\iy$ and
so $p\sim_\lam x_\iy$. Together with $x_1\sim_\lam p$ this implies
$x_1\sim_\lam x_\iy$ as desired.


Let us show that $\sim_\lam$-classes are pairwise unlinked. Indeed, otherwise
there exist 4 distinct points $x_i, i=1, \dots, 4$ such that $x_1\sim_\lam x_3,
x_2\sim_\lam x_4$ and $x_1, x_3$ are not $\sim_\lam$-equivalent. However then
the location of the points $x_1, \dots, x_4$ on the circle implies that
$\om$-continua $K'$ (containing $x_1, x_3$) and $K''$ (containing $x_2, x_4$)
are non-disjoint and hence their union is an $\om$-continuum containing all 4
points $x_1, \dots, x_4$ and showing that in fact they are all equivalent, a
contradiction.


Suppose next that $x_i\sim_\lam y_i$ and $(x_i,y_i)\to (x_{\infty},y_{\infty})$
in $\ucirc\times \ucirc$. We must show that $x_{\infty}\sim_\lam y_{\infty}$.
Assume that $x_{\infty}\ne y_{\infty}$; since classes are closed we may also
assume that $x_i\ne x_\iy, y_i\ne y_\iy$ for any $i$. Let $K_i$ be
$\omega$-continua containing $x_i$ and $y_i$. We may assume that $\lim
K_i=K_\infty\subset\Lstar$ exists (in the sense of Hausdorff metric). Since
classes are closed and pairwise unlinked and $x_\iy\ne y_\iy$, we may also
assume that all $K_i$ are pairwise disjoint and $x_i, y_i$ are such that the
chord $\ol{x_iy_i}$ is disjoint from the chord $\ol{x_\iy y_\iy}$ for any $i$
which implies that the convex hull of each $\eq{x_i}$ is disjoint from
$\ol{x_\iy y_\iy}$. Each $K_i$ is contained in the convex hull of $\eq{x_i}$.
Since the convex hulls of the $\eq{x_i}$'s are disjoint, then $K_{\infty}$ must
be a leaf in $\lam$.   So $x_{\infty}\sim_\lam y_{\infty}$, and $\sim_\lam$ is
a lamination.

Let us show that $\sim_\lam$ and $R_\lam$ are the same.  Since $\sim_\lam$ is
compatible with $\lam$, $R_\lam$ is finer than $\sim_\lam$.  We show that
$\sim_\lam$ is finer than any lamination compatible with $\lam$. Let $\sim$ be
a lamination   compatible with $\lam$. Then for any two points $x, y\in \ucirc$
with $x\sim_\lam y$ we have to prove that $x\sim y$. Since $x\sim_\lam y$ then
there exists an $\om$-continuum $K\subset \Lstar$. To proceed we first extend
the equivalence $\sim$ onto $\disk$ by declaring two points $u, v\in \disk$
equivalent if and only if for some class $A$ we have $u, v\in \ch(A)$. Clearly
the new equivalence $\approx$ is an extension of $\sim$. Set $Z=\disk/\approx$
and let $\pi:\disk\to Z$ be the corresponding quotient map. Let us show that
$\pi(K)=\pi(K\cap \ucirc)$. Indeed, if $x\in K\setminus \ucirc$ then $x$
belongs to the appropriate leaf $\ell$ and by Lemma~\ref{contain} and endpoint
$a$ of $\ell$ belongs to $K$. Since $\pi(a)=\pi(x)$ we see that $\pi(x)\in
\pi(K\cap \ucirc)$ and so $\pi(K)=\pi(K\cap \ucirc)$. Since $K$ is an
$\om$-continuum then $K\cap \ucirc$ is countable, hence $\pi(K)=\pi(K\cap
\ucirc)$ is at most countable and therefore a point. Thus, $\pi(x)=\pi(y)$ and
$x\sim y$ as desired.

\end{proof}

\section{Non-periodic Critical Class}\label{nonper}

By Theorem~\ref{thu} (see \cite[Proposition II.4.5]{thu85}) for a critical leaf
$\ell_\ta$ we can construct an invariant geometric lamination $\lam^\ta_{\iy}$.
Slightly abusing the language let us call a critical leaf $\ell_\ta$
\emph{periodic} if $\si(\ta)$ is periodic and \emph{non-periodic} otherwise.
Observe that $\ell_\ta$ is periodic if and only if either $\ta$ or $\ta'=\ta +
1/2$ is periodic. In this
section we show that if $\ell_\ta$ is non-periodic then the lamination
$\sim_{\lam^{\ta}_{\iy}}$ constructed in the previous section is
non-degenerate. The remaining part of the paper is concerned with the case that $\ell_\ta$
is periodic.

Note that for an invariant geometric lamination $\lam$, we can extend the map
$\si$ over $\complex$ as follows.  First extend $\si$ over $\complex\setminus
\disk$ by sending the point $(r,\theta)$ in polar coordinates to the point
$(r^2, \si(\ta))$. Next extend linearly over $\Lstar$ and subsequently over
gaps, by mapping for a gap $Q$ the barycenter of $Q\cap \ucirc$  to the
barycenter of its image and by mapping the line segment from the barycenter of
a gap to a point on its boundary linearly onto the corresponding line segment
in its image. We will denote this extended map by $\Si$.

Note that $\Si$ is the composition of a monotone map $m:\complex\to X$ and an
open and light map $g:X\to \complex$. Since open maps are confluent \cite[1.5]{whyb37}, $\Si$ is
confluent (i.e., for each continuum $K\subset\complex$ and each component $C$
of $\Si^{-1}(K)$, $\Si(C)=K$).

Our ``Test for Degeneracy" (Theorem~\ref{thnonper}) applies to
the geometric lamination $\lam^\ta_\iy$ constructed by pulling back a critical
leaf $\ell_\ta$.  By\cite[Proposition II.4.5]{thu85}, $\lam^\ta_\iy$ is a
geometric lamination and its leaves can only meet at points of $\ucirc$.  To
prove the theorem, we study how the leaves of $\lam^\ta_\iy$ can
meet under the assumption that $\ell_\ta$ is non-periodic.

\smallskip

\noindent\textbf{Non-Periodic Assumption.} For the rest of this section assume
that $\lam=\lam^\ta_\iy$ is generated by pulling back a {\em non-periodic}
critical leaf $\ell_\ta$.

\smallskip

 The following series of lemmas is trivial but
important.

\begin{lem} \label{preimageleaves} If two leaves of $\lam$ meet, at least
one is a limit leaf.
\end{lem}

\begin{proof}  Since $\ell_\ta$ is non-periodic, no precritical leaves
can meet.
\end{proof}

\begin{lem}  \label{threeleavesmeet}
At most three leaves of $\lam$ can meet at a point, and if three
leaves do meet, the middle leaf is a precritical leaf.
\end{lem}

\begin{proof} Suppose that four leaves $\ell_1,
\ell_2, \ell_3, \ell_4$ of $\lam$ meet at a point $a\in \ucirc$, and
assume that they are numbered so that the angle between $\ell_1$ and
$\ell_4$ taken in the positive direction is less than $\pi$ and the
leaves $\ell_2, \ell_3$ are contained in this angle.  By
Lemma~\ref{preimageleaves}, at least one of $\ell_2$ or $\ell_3$,
without loss of generality say $\ell_2$, is a limit leaf. Then it is
a limit of precritical leaves from at least one side.  This would
require a sequence of precritical leaves to meet, or to intersect
$\ell_1$ or $\ell_3$ but not in $\ucirc$, both of which are
impossible. So no more than three leaves of $\lam$ intersect at one
point, and the middle one is a precritical leaf by the above
argument.
\end{proof}

\begin{lem} \label{omegacnt}  If $K$ is an $\omega$-continuum in
$\lam$, then $K$ can meet only countably many leaves of $\lam$.
\end{lem}

\begin{proof}  Let $K$ be an $\omega$-continuum in $\lam$.  By definition,
$K\cap\ucirc$ is countable.  By Lemma~\ref{contain}, if $K$ meets more than one
leaf, then $K$ meets $\ucirc$ at an endpoint of each leaf it meets. Since at
most three leaves of $\lam$ can meet at a point, $K$ can meet only countably
many leaves, for otherwise its intersection with $\ucirc$ would be uncountable.
\end{proof}

\begin{thm}[Test for Non-Degeneracy]\label{thnonper}
Let $\ell_\ta$ be a critical diameter, and let $\lam^\ta_\iy=\lam$ be the
corresponding geometric lamination. If $\ell_\ta$ is not periodic then
$\sim_{\lam}$ is a non-degenerate invariant lamination.
\end{thm}

\begin{proof} Let $x\sim_\lam y$ and $K\subset \Lstar$ be an $\om$-continuum
containing $x$ and $y$. Then $\Si(K)$ is an $\om$-continuum containing
$\si(x)$ and $\si(y)$. Hence $\si(\eq{x})\subset \eq{\si(x)}$. Let
$z\in\eq{\si(x)}$ and let $H$ be an $\omega$-continuum containing $z$ and
$\si(x)$. Since $\Si$ is confluent and $\si$ is 2-to-1, the component $C$
of $\Si^{-1}(H)$ which contains $x$ is an $\om$-continuum with $\Si(C)=H$.
So, $\si(\eq{x})=\eq{\si(x)}$ and $\si$-images of
$\sim_\lam$-classes are $\sim_\lam$-classes.

It remains to show that $\sim_\lam$ is non-degenerate. We achieve
this either by finding an uncountable collection of pairwise
disjoint leaves, or by carefully examining the boundary of a
periodic gap. There are two cases: either the critical leaf
$\ell_\ta$ is isolated in $\lam$ or it is not.

{\bf Case 1.}  Suppose first that $\ell_\ta$ is not isolated in
$\lam$. Then it is a limit on at least one side and hence, by the
symmetry of the construction, from both sides.   Since limit leaves
are limits of precritical leaves, it is a limit of other leaves of
the pre-lamination $\lam^\ta=\cup_n \lam_n$ from both sides. Hence,
between any two leaves of $\lam^\ta$, there is another leaf of
$\lam^\ta$. By Lemma~\ref{threeleavesmeet}, it follows that there is
an uncountable collection of disjoint leaves in the closure $\lam$
of $\lam^\ta$. By  Lemma~\ref{omegacnt}, $\sim_\lam$ is
nondegenerate.


 {\bf Case 2.} Consider next the case when $\ell_\ta$ is isolated. Then there exist
two gaps $G_1$ and $G_2$ (symmetric about $\ell_\ta$) such that
$G_1\cap G_2=\ell_\ta$.  Let $G=G_1\cup G_2$.  We will consider
three cases: either every leaf of $\bd G$ is a precritical leaf or
not, and in the latter case, either every leaf of $\bd G$ is a limit
leaf, or not.

{\bf Case 2a.} All leaves in $\bd G$ are precritical leaves.  Then
either $G_1$ or $G_2$ maps onto itself under the first iterate
$\si^k$ which takes a precritical leaf in $\bd G_i$ to $\ell_\ta$.
Renaming, if needed, $G_1$ maps onto itself.  Since precritical
leaves are disjoint, and $G_1$ is periodic, pulling $G_1$ back
through its orbit, we see that the leaves of $\bd G_1$ are infinite
in number and pairwise disjoint.  Hence, $G_1\cap\ucirc$ is a Cantor
set.  By Lemma~\ref{contain}, any two points of $G_1\cap\ucirc$
which are not the endpoints of a leaf cannot be joined by an
$\omega$-continuum, so $\sim_\lam$ is non-degenerate.

{\bf Case 2b.} All leaves in $\bd G$ are limit leaves. Then every
leaf of $\bd G$ must be a limit leaf from exactly one side.  Pulling
$G$ back, we see that between any two preimages of $G$ there are
limit leaves, so by Lemma~\ref{threeleavesmeet}, pre-images of $G$
are disjoint. Moreover, limit leaves are actually limits of $G$
(since $\ell_\ta$ is within $G$). As in Case 1, because of the limit
leaves, we have pre-images of $G$ between any two pre-images of $G$.
Since the pre-images of $G$ are disjoint, this gives us an
uncountable collection of leaves in $\lam$.  So again, $\sim_\lam$
is nondegenerate.

{\bf Case 2c.}  There are both precritical leaves and
non-precritical limit leaves in $\bd G$.  For each precritical leaf
$\ell_i$ in $\bd G$, there is a preimage of $G$ sharing that leaf
with $G$.  Let $G_\iy$ be the component of $G$ in the union of all
preimages of $G$ in $\lam$.  Then $\bd G_\iy$ contains only limit
leaves.  Moreover, because $\lam$ does contain limit leaves, $G_\iy$
is not all of $\lam$.  Because of the limit leaves in preimages of
$G_\iy$, we can now argue as in Case 2b that between any two
preimages of $G_\iy$ there is another preimage of $G_\iy$. Since the
pre-images of $G_\iy$ are disjoint, this gives us an uncountable
collection of leaves in $\lam$.  So again, $\sim_\lam$ is
nondegenerate.
 \end{proof}

\section{Renormalization of Dendrites}\label{per}

A dendrite is a locally connected continuum which does not contain a
subset homeomorphic to the unit circle. Let $X$ be a dendrite. Let
$[x, y]$ be the (unique) closed arc in $X$ connecting $x$ and $y$
(similarly we define open and semi-open arcs $(x, y), [x, y), (x,
y]$).  It is well-known that every subcontinuum of a dendrite is a
dendrite and that dendrites have the fixed point property. (See
\cite[Chapter X]{nad} for further results about dendrites.) A set
$A\subset X$ is said to be \emph{condense} in $X$ if $A$ is
\emph{dense} in each \emph{continuum} $K\subset X$. The notion has
been introduced in \cite{BOT05} in a very different setting (in
\cite{BOT05} we study, for some compact and $\sigma$-compact spaces,
how big the set of points with exactly one preimage should be to
guarantee that the map is an embedding or a homeomorphism). Also,
given a closed set $P\subset X$ let the \emph{continuum hull} $T(P)$
of $P$ be the smallest continuum in $X$ containing $P$ (in
particular, if $P=\{x, y\}$ is a two-point set then $T(x, y)=[x,
y]$, and, more generally, if $P$ is finite than $T(P)$ is a
\emph{tree}, i.e. a one-dimensional branched manifold). For any
connected topological space $Y$ a point $y\in Y$ is said to be a
{\em cutpoint} of $Y$ iff $Y\setminus \{y\}$ is not connected and an
\emph{endpoint} of $Y$ otherwise. Also, the number of components of
$Y\setminus \{y\}$ is said to be \emph{valence} of $y$ (in $Y$) and
points of valence greater than two are said to be \emph{branch
points} or \emph{vertices} of $X$. Given a map $f:X\to X$, a set $Z$
is said to be \emph{periodic (of period $m$)} if $Z, f(Z), \dots,
f^{m-1}(Z)$ are pairwise disjoint and $f^m(Z)\subset Z$. Now we are
ready to prove Theorem~\ref{TH AlphaFixPt}.

\begin{thm}\label{TH AlphaFixPt}
Let $f:X \to X$ be a continuous self-mapping of a dendrite $X$ with
no wandering continua and no identity return. Then $f$ has non-fixed
critical cutpoints. Moreover, if $f$ is a finite-to-one map with
finitely many critical points then it has fixed cutpoints and for
all $n$, there exists no interval $I$ such that $f^n|_I$ is a
one-to-one map  (in particular, all preimages of critical points are
condense in $X$). Finally, if $f$ has exactly one critical point $c$
then $f$ has a fixed cutpoint $a\in (c, f(c))$.
\end{thm}

\begin{proof} Let us prove the first claim. In the interval case
it is obvious (if $f:I\to I$ is an interval map without critical points then it
is easy to see that either $f$ has a wandering interval or $f$ has an interval
of identity return). Hence there are no periodic intervals on which $f$ would
not have a critical point; we often use this argument in the future.

If the map $f$ collapses an interval, then there are non-fixed
critical cutpoints. Hence we may assume that the closed set $C_f$ of
all critical points of $f$ is totaly disconnected.  Assume that all
critical cutpoints of $f$ (if any) are fixed. Now, if there are at
least two critical cutpoints then we can choose critical points
$c_1\ne c_2$ so that $(c_1, c_2)$ contains no critical points of $f$
(just consider $[a, b]$ with $a, b\in C_f$ and choose an arc in it
complementary to $C_f\cap [a, b]$). Then $f:[c_1, c_2]\to [c_1,
c_2]$ is a homeomorphism, a contradiction. So we may assume that
there is at most one fixed critical cutpoint. Now, if $f$ is not a
homeomorphism then there exist points $x\ne y$ such that
$f(x)=f(y)$. Then there must exist a critical cutpoint $c\in (x,
y)$. Thus the only two cases to consider are (a) when $f$ is a
homeomorphism, and (b) when $f$ has a unique critical cutpoint $c$,
and $c$ is a fixed point.

Let $v$ be a fixed point of $f$ in  case (a) and $c$ in  case (b).
Let $\{I_\al\}_{\al\in A}$ be the family of closures of components
of $X\setminus \{v\}$. Then for each $\al\in A$ there exists $\be\in
A$ such that the restriction $f|_{I_\al}$ is a homeomorphism into
$I_\be$. Since there are no wandering continua, $f^m(I_\al)\subset
I_\al$ for some $\al\in A$ and $m>0$. Choose a point $x\in
I_\al\setminus \{v\}$ and consider the interval $[v, x]\cap [v,
f^m(x)]=[v, a']=I$. Consider two cases depending on the location of
$f^m(a')$. If $f^m(a')\in I$ then $f^m$ maps $I$ into itself
homeomorphically which is impossible. Suppose that $f^m(a')\nin I$
and consider the component $J$ of $X\setminus \{a'\}$ containing
$f^m(a')$. Denote the retraction of $X$ onto $J$ (which maps
$X\setminus J$ to $a'$) by $R$ , consider the map $g=R\circ f^m:J\to
J$, and let $b$ be a fixed point of $g$. It follows that $b\ne a'$,
hence $b$ in fact is a fixed point of $f^m$ too. Since $b\ne v$ we
see that $[v, b]$ is a non-degenerate interval mapped onto itself by
$f^m$ homeomorphically, a contradiction. Hence there exist non-fixed
critical cutpoints of $f$.

Now  we restrict ourselves to maps $f$ with finitely many critical points.
 Under our assumptions we can show that $f$ has a fixed cutpoint. Indeed,
assume otherwise. Then all fixed points of $f$ are endpoints of $X$.
Let $b$ be a fixed point of $f$.  It is easy to see that in a
dendrite with finitely many critical points, and endpoint cannot be
a critical point. Hence, there exists a connected neighborhood
$U=U_b$ of $b$ on which $f$ is one-to-one.
 Observe that if now
$U$ contains another fixed point $s$ of $f$ then $f: [s, b]\to [s,
b]$ is a homeomorphism which contradicts the assumptions. So fixed
points form a closed set of isolated points, hence there are
finitely many of them.

Denote the set of all fixed points of $f$ by $B$. Let $b\in B$ and
let $U_b$ be a neighborhood chosen as above. Choose an interval
$I\subset U_b$ with one endpoint $b$ and consider the interval
$f(I)\cap I=[b, d]$. Then choose a point $y\in I$ so that $f(y)=d$.
Since $f$ has no wandering intervals it follows that $[b, y]\subset
[b, d]$ (otherwise $[b, y]$ maps homeomorphically into itself and
has a wandering interval). In other words, the point $y$ is repelled
away from $b$ by $f$. Since there are no fixed points in $(b,y]$ it
implies that all points of $(b, y]$ are repelled away from $b$. We
can choose $y=y_b$ very close to $b$ so that $y$ is not a vertex of
$X$ (it is known that $X$ can have no more than countably many
vertices \cite[10.23]{nad}). Denote by $V_b$ the component of
$X\setminus \{y_b\}$ containing $b$. Clearly, this can be done for
all fixed points of $f$ so that for distinct fixed points $b, q$ the
neighborhoods $V_b, V_q$ are disjoint and, moreover, their
$f$-images are disjoint. Consider now the dendrite $Y=X\setminus
\cup_{b\in B} V_b$ and define the retraction $R:X\to Y$ (by
collapsing all points of every $V_b, b\in B$ into $y_b$ and keeping
the identity map on $Y$). Then define the map $g=R\circ f: Y\to Y$.
By the construction no point $y_b$ is $g$-fixed. On the other hand,
$B$, the set of all $f$-fixed points, is disjoint from $Y$. Hence
$g$ is a fixed-point-free map on the dendrite $Y$, a contradiction.
This implies that $f$ \emph{must} have at least one fixed cutpoint.

Let us prove that for all $n$, there exists no such interval $I$
that $f^n|_I$ is a 1-to-1 map.  It then would follow that the set of
pre-critical points is condense. Suppose otherwise and assume that
$I$ is closed. Consider the orbit $Q''=\cup^\iy_{j=0} f^j(I)$ of
$I$. Since $I$ is not wandering, there exist $k$ and $k+l$ such that
$f^k(I)\cap f^{k+l}(I)\ne \0$. Then we can consider the set
$Q=\ol{\cup^\iy_{j=0} f^{jl}(f^k(I))}$. It follows that $Q$ is a
subdendrite of $X$ such that $f^l(Q)\subset Q$. Observe that all the
assumptions of the theorem hold for $f^l|_Q$, hence the results of
the previous paragraph apply to $f^l|_Q$ and there exists a
$f^l$-fixed cutpoint $a$ in $Q$. By the construction it follows that
some power of $I$ contains $a$, so we may assume from the very
beginning that $a\in I$. Moreover, replacing $f$ by $f^l$ and $X$ by
$Q$ we may assume that $a$ is a fixed cutpoint of $X$ and $I=[a, b]$
is an interval such that for all $n$, $f^n|_I$ is one-to-one.

Let us show that then we may assume that there exists $r$ such that
an image of a small interval $Z=[a, d]\subset I$ maps back over
itself by $f^r$ so that points are repelled away from $a$ within
$Z$. Indeed, suppose there were no such interval $Z$.  Then the
successive images of $I$ would meet only in the fixed cutpoint $a$.
Hence a small interval bounded away from $a$ in $I$ would wander, a
contradiction. Hence, we may assume that there is an $r$ such that
$f^r(Z)\supset Z$. Consider $Z_\infty=\cup^\iy_{i=0} f^{ri}(Z)$.
Then the assumptions of the theorem apply to the dendrite
$Q'''=\ol{Z_\infty}$ and $f^r:Q'''\to Q'''$, and imply that there
exists a critical cutpoint of $f^r|_{Q'''}$ and that $f^r|_{Q'''}$
is not one-to-one.  Note that because $Z=[a,d]$ maps over itself
one-to-one under $f^r$, $f^r$ is one-to-one on $Z_\infty$. Since
closure can only introduce endpoints of $Q'''$ to $Z_\infty$, there
are endpoints $x'\ne y'\in Q'''$ such that $f^r(x')=f^r(y')$. But
then by continuity, there are non-endpoints $x\not= y\in
Z_\infty\subset Q'''$, such that $f^r(x)=f^r(y)$, a contradiction
with $f^r$ being one-to-one on $Z_\infty$.

To prove the rest of the theorem we prove a series of claims
assuming that $f$ has a unique critical point $c$. By the first
claim $c$ is a cutpoint and $f(c) \ne c$; set $A=[c, f(c)]$. Then
$f^2(c)\nin A$ (otherwise, there is a fixed critical cutpoint in
$(c,f(c))$, contradicting our assumptions). Consider the interval
$f(A)=[f(c), f^2(c)]$ and show that the point $f(c)$ cannot belong
to the interval $[c, f^2(c)]$. Suppose otherwise. Then $A$ and
$f(A)$ are concatenated (have only $f(c)$ in common) and $f|_{[c,
f^2(c)]}=f|_{A\cup f(A)}$ is a homeomorphism which implies that
$f(A)$ and $f^2(A)$ are concatenated (have only $f^2(c)$ in common),
etc. By induction all the images of $A$ form a concatenated sequence
of intervals mapped on each other homeomorphically - i.e., in a
sequence of intervals $A, f(A), f^2(A), \dots$ the consecutive
intervals have only one endpoint $f^i(c)$ in common.  However, then
a small  subinterval of $A$ is a wandering continuum, a
contradiction. Hence $A$ and $f(A)$ have a non-degenerate
intersection.

Let $[f(c), d]=A\cap f(A), d\ne f(c)$. Let $R$ be the monotone
retraction of $X$ onto $A$. Consider a map $g=R\circ f:A\to A$.
Denote by $a$ a fixed point of $g$. Then $a\ne c$ and $a\ne f(c)$.
Let us show that in fact $f(a)=a$. Indeed, suppose otherwise. Then
$f(a)\nin A$ and hence $f(a)\in [d, f^2(c)]$. This implies that
$R(f(a))=d=a$ and so the interval $f([c, d])=[f(c), f(d)]$ contains
$d$. Choose a point $u\in [d, c]$ so that $f(u)=d$ and set $B=[u,
d]$. Then the interval $f(B)=[d, f(d)]$ is concatenated with the
interval $B$ at their common endpoint $d$, and applying the same
arguments as before we can see, that this kind of dynamics is
impossible under the assumption that $f$ has no wandering continua.
Hence $f(a)=a\in (c,f(c))$ is a fixed cutpoint as desired.
\end{proof}

Lemma~\ref{TH PerPropts} shows that to study the remaining case when a critical
diameter has a periodic endpoint we need to study maps of dendrites.

\begin{lem}\label{TH PerPropts}
Let $\sim$ be a non-degenerate invariant lamination with a critical class
$C$. Then $f_\sim:J_{\sim} \to J_{\sim}$ has exactly one critical point which
is the image of $C$ under the quotient map $p$. Moreover, if
$C$ contains a preperiodic point of $\si$ then $J_{\sim}$ is a dendrite.
\end{lem}

\begin{proof} We can find sequences $x_i\to x, x'_i\to x'$ with $x\ne x'\in C$
so that $x_i\not \sim x'_i$ for any $i$ and $\si(x_i)=\si(x'_i)$. Indeed,
choose $x, x'\in C$ so that $\si(x)=\si(x')$. Points of $\ucirc$ separated by the chord
connecting $x$ and $x'$ cannot be $\sim$-equivalent unless they belong to $C$,
hence we can choose the desired
sequences. If $p(C)=c$ then $p(x_i)\to c, p(x'_i)\to c, p(x_i)\ne p(x'_i)$ and
$f_{\sim}(p(x_i))=f_{\sim}(p(x'_i))$ which implies that $c$ is a critical point
of $f_\sim$. On the other hand, if $Q$ is a non-critical class then $p(Q)$ is
not a critical point of $f_\sim$. Indeed, suppose otherwise. Then we can choose
a sequence of pairs of points $y_i, y'_i\to p(Q)$ in the quotient space
$J_\sim$ so that $f_\sim(y_i)=f_\sim(y'_i)$. Then we can choose two converging
sequences of points $z_i\ne z'_i\in \ucirc$ such that $p(z_i)=y_i, p(z'_i)=y'_i$ and
$\si(z_i)=\si(z'_i)$. Let $z_i\to z, z'_i\to z'$. Then $\si(z)=\si(z')$ and
$z\ne z'$ (the latter follows from the fact that the arcs between $z_i, z'_i$
are actually semicircles). On the other hand, $y_i\to p(Q), y'_i\to p(Q)$ and
hence $z, z'\in Q$, a contradiction with $Q$ being non-critical. Hence if
$\sim$ has a critical class $C$ then $f_\sim:J_\sim \to J_\sim$ has a unique
critical point $c=p(C)$.

Let $\hf_\sim=\hf: \sphere \to \sphere$ be an orientation preserving
branched covering map of degree $2$ extending $f$ (see
\cite{kiwi97}); then its unique finite critical point must coincide
with $c$. If $J_\sim$ is not a dendrite then by Proposition~\ref{TH
NoWandDom or IdRet} we may assume that there exist a bounded
complementary to $J_\sim$ domain $H$ and a number $m$ such that
$\hf^m(H)=H$. Since $c$ is a unique critical point of $\hf$, we see
that $\hf^m|_H$ is a homeomorphism. Since there are no wandering
continua or identity returns for $f_\sim$ then in fact $\hf^m|_{\bd
H}$ is an irrational rotation. Consider the set
$p^{-1}(\cup_{i=0}^{m-1} \bd f^i(H))$. Then by \cite[Lemma
18.8]{miln}, since  $\si$ is locally expanding on $\ucirc$ and the
fact that the boundary of $H$ is infinite imply that for some $i$
the restriction $\si|_{p^{-1}(\bd f^i(H))}$ \emph{is not}
one-to-one. Since downstairs $f|_{f^i(\bd H))}$ \emph{is}
one-to-one, this means that the critical class $C$ is contained in
$p^{-1}(\bd f^i(H))$ and hence $c\in \bd(f^i(H))$, a contradiction
to $c$ being preperiodic.
\end{proof}

Now we consider the most specific cases in this subsection. As in
Lemma~\ref{TH PerPropts}, we work with induced maps of laminations,
sometimes with the extra assumption that the critical point is
periodic. However we only rely upon the dynamical properties of
induced maps. Thus, given a dendrite $J$ we denote by $\ty(J)$ the
family of all $2$-to-$1$ branched covering self-mappings of $J$
which have no wandering continua and no identity return.  Note that
all such maps are open, hence confluent.  Denote the family of all
such maps by $\ty$ if the dendrite is not fixed. For $f\in \ty$ we
denote by $c_f=c$ its unique critical point.

By Proposition~\ref{TH NoWandDom or IdRet} and Lemma~\ref{TH
PerPropts} induced maps of laminations belong to $\ty(J_\sim)$ if
$J_\sim$ is a dendrite (i.e., if $C$ contains a preperiodic point).
We will show that the unique critical point of $J_\sim$ cannot admit
a certain type of dynamics, called a {\em snowflake} and defined
below (see Lemma~\ref{le-version}.  For this purpose we introduce
the notion of a {\em rotational renormalization} $F_1$ of $f$ (see
Lemma~\ref{le-Rinf} and paragraphs following).

 Suppose that $f\in \ty(J)$. For $x\in J\setminus c$ let $x'$
be the unique point of $J$ such that $x'\ne x, f(x')=f(x)$, and set
$c'=c$. Then the map $x\mapsto x'$ is a continuous involution of
$J_\sim$. From now on given a point $z\in J$ (a set $A\subset J$) by
$z'$ ($A'$) we mean the image of $z$ (of $A$) under this involution.
Clearly, for any $x$ we have $c\in [x, x']$. Also, the family of all
$f\in \ty(J)$ with $c$ is periodic is denoted by $\tp(J)$ or just
$\tp$ (if the dendrite is not fixed).

Fix $f\in \ty(J)$. By Theorem~\ref{TH AlphaFixPt} there exists a
fixed cutpoint $a\in (c, f(c))$. Denote the component of $J\setminus
\{a\}$ containing $c$ by $K$; then $f(c)\nin \ol{K}$. This implies
that the fixed cutpoint of $f$ is unique. Indeed, suppose that $b\ne
a$ is another fixed cutpoint of $f$. Then $c\in [a, b]$ for
otherwise $f: [a, b]\to [a, b]$ is a homeomorphism. Denote by $K'$
the component of $J\setminus \{b\}$ containing $c$. Consider all
other components of $J\setminus \{b\}$. The fact that $f$ is a local
homeomorphism at $b$ implies that these components of $J\setminus
\{b\}$ have images disjoint from $K'$. Since $f$ has no wandering
continua we can find a component $H$ of $J\setminus \{b\}$
homeomorphically mapping into itself by some $f^m$ which is
impossible by Theorem~\ref{TH AlphaFixPt}. The unique fixed cutpoint
of $f$ is denoted by $a_f=a$.

Now we need the notion of a pullback. Given a map $f\in
\ty(J)$, a continuum $Q\subset J$, a point $x\in J$ and a number $n$ such that
$f^n(x)\in Q$ we call the component $V$ of $f^{-n}(Q)$ containing $x$ the
\emph{pullback of $Q$ along $x, \dots, f^n(x)$}. Since $f\in \ty(J)$ is a
branched covering map and $J$ is a dendrite, it follows that the map $f^n$ maps
$V$ onto $Q$ as a branched covering map, and the degree of $f^n|_V$ equals
$2^s$ where $s$ is the number of times the images $V, f(V), \dots, f^{n-1}(V)$
of $V$ contain $c$ as a cutpoint. This notion is usually used for rational maps, but it can
also be defined in our setting.

By Theorem~\ref{TH AlphaFixPt} $f(c)\ne a$. Then there exists the least $m_f=m$
with $f^m(c)\in K$ (otherwise the continuum hull $T(a\cup \ol{\orb(c)})$ is a
non-degenerate dendrite mapped into itself homeomorphically, a contradiction to
Theorem~\ref{TH AlphaFixPt}). Clearly, $c\in (a, a')$. Consider the closure $R$
of the component of $J\setminus \{a, a'\}$ containing $c$. In Lemma~\ref{le-R1}
we describe the set $R_1=R\cap f^{-m}(R)$ of all points of $R$ mapped back into
$R$ by $f^m$.

\begin{lem}\label{le-R1} One of the following two possibilities holds.

\begin{enumerate}

\item If $f^m(c)\nin R$ then $R_1=V\cup V'$ where $V$ and $V'$ are
disjoint continua, $f^m(V)=f^m(V')=R$ and both $f^m|_V$ and
$f^m|_{V'}$ are homeomorphisms onto $R$.

\item If $f^m(c)\in R$ then $R_1$ is a dendrite and $f^m:R_1\to R$ is a
$2$-to-$1$ branched covering  map whose unique critical point is
$c$.

\end{enumerate}

\end{lem}

\begin{proof} Clearly, $[a, c]\cap f^m[a, c]=[a, d]$ with some $d$,
and there is a point $u\in [a, d]$ with $f^m(u)=d$ and a point
$u_1\in [a, u]$ with $f^m(u_1)=u$. The arc $[a, u]$ ``rotates''
about $a$ and comes back onto $[a, d]$ after $m$ steps. In other
words, $[a, u]$ ``sweeps'' through the germs (at $a$) of all
components of $J\setminus \{a\}$. Observe also that $[a, d]\subset
[a, c]\subset [a, a']$.  Let $V$ be the pullback of $R$ along $u_1,
f(u_1), f^m(u_1)$. Then $f^m(V)=R$. Let us show that $V\subset R$.
Indeed, suppose that there is a point $y\in V\setminus K$. Take a
point $z\in [a, y]$ close to $a$. Then since $f$ is a local
homeomorphism at $a$ and $f^m([a, u_1])=[a, u]$ we see that
$f^m(z)\nin K$ and hence $f^m(z)\nin R$, a contradiction. For $y\in
V\setminus K'$ we get a similar conclusion. So, $ V\subset R$ is the
component of $f^{-m}(R)$ containing $u_1$, and $V'\subset R$ is the
component of $f^{-m}(R)$ containing $u'_1$.

Set $R_1=V\cup V'$. Since both $V$ and $V'$ are pullbacks of $R$
then either $V=V'$, or $V\cap V'=\0$. Suppose that $f^m(c)\nin R$.
Then $V\cap V'=\0$ since otherwise $c\in [a, a']\subset V=V'$ and
hence $f^m(c)\in R$, a contradiction. In this case $f^m$ maps $R_1$
onto $R$ as a 2-to-1 covering map. Now, suppose that $f^m(c)\in R$.
Then as above it follows that $c\in V\cap V'$. So in this case
$V=V'$ is the pullback of $R$ along $u_1, f(u_1), \dots,
f^m(u_1)=u$, and $f^m$ maps $R_1$ onto $R$ as a 2-to-1 branched
covering map with the critical point $c$.

\end{proof}

So, $R_1\subset R, f^m(R_1)=R$; iterating this, we consider the sets
$R_i=\{x:x\in R, f^m(x)\in R, \dots, f^{im}(x)\in R$, and the set
$R_\iy=\{x: f^{jm}(x)\in R, j\ge 0\}=\cap_i R_i$. Thus, $R_\iy$ is
the set of all points whose $f^m$-orbits are contained in $R$. In
Lemma~\ref{le-Rinf} we relate the local properties at $a$ and the
orbit of $c$. Observe, that by the above, locally at $a$ there are
always $m>1$ small semiopen arcs (for example, $(a,u_1],
(a,f(u_1)],\dots$) which exclude $a$ and are cyclically permuted by
$f$ so that the first arc maps over itself under $f^m$ in a
repelling fashion, and each component of $J\setminus \{a\}$ contains
exactly one of these arcs (thus, at $a$ the map is a \emph{local
rotation (of local period $m$)}).

\begin{lem}\label{le-Rinf} The set $R_\iy$ is a continuum if and only if
the $f^m$-orbit $\orb_{f^m}(c)$ of $c$ is contained in $R$; in this
case $F_1=f^m|_{R_\iy}\in \ty(R_\iy)$. In particular, the unique
critical point $c$ of $F_1$ is not fixed and hence if $c$ is
periodic then its period is not equal to the local period at $a$.
\end{lem}

\begin{proof} Suppose that $\orb_{f^m}(c)\not \subset R$;
then $R_\iy$ is not connected since otherwise $[a, a']\subset
R_\iy$, and the orbit of $f^m(c)$ is contained in $R$. Hence, $c\in
[a, a']\subset R_\iy$, a contradiction. Suppose now that
$\orb_{f^m}(c)\subset R$ and show that then $R_\iy$ is connected.
Indeed, by induction it is easy to see that in this case the entire
arc $[a, a']$ is mapped into $R$ by all powers of $f^m$. By
Lemma~\ref{le-R1} $R_1$ is the $f^m$-pullback of $R$ along $c,
f^m(c)$. Since $f^{2m}(c)\in R_1$, $R_2$ is the $f^m$-pullback of
$R_1$ along $f^m(c), f^{2m}(c)$, i.e. $R_2$ is the $f^m$-pullback of
$R$ along $c, f^m(c), f^{2m}(c)$. Continuing by induction we see
that $R_i$ is the pullback of $R$ along $c, f^m(c), \dots,
f^{im}(c)$. All $R_i$'s are continua, and since $R_\iy=\cap_i R_i$
we see that $R_\iy$ is a continuum too.

Let us show that then $F_1=f^m|_{R_\iy}\in \ty(R_\iy)$. Indeed,
clearly $R_\iy$ is a dendrite and $F_1$ has no wandering continua or
identity return. Since $R_\iy\subset R_1$ is symmetric in the sense
that $R'_\iy=R_\iy$ then $F_1|_{R_\iy}$ is 2-to-1, and it follows
from the definition that $f^m(R_\iy)=R_\iy$. Thus, $F_1\in
\ty(R_\iy)$. By Theorem~\ref{TH AlphaFixPt} the critical point $c$
of $F_1$ \emph{cannot be fixed}, and  this implies that if $c$ is
periodic then its period is not equal to the local period $m$ at
$a$.
\end{proof}

If $R_\iy$ is connected, we call $F_1$ a \emph{rotational
renormalization (of generation $1$)} of $f$; the $F_1$-orbit of $c$
is not a fixed point. Apply to $F_1$ the same construction; then
either $F_1$ is rotationally renormalizable or not. If it is, we
denote its rotational renormalization $F_2$ and call $F_2$ the
\emph{rotational renormalization of $f$ of generation $2$}. As
above, by Theorem~\ref{TH AlphaFixPt} the $F_2$-orbit of $c$ is not
a fixed point. The process of constructing rotational
renormalizations $F_n$ of $f$ can continue as long as we get
rotationally renormalizable maps; by Theorem~\ref{TH AlphaFixPt} if
on the step $n$ we get a map $F_n$, the $F_n$-orbit of $c$ is not a
fixed point. The $F_n$-orbit of $c$ will be called the
\emph{rotational renormalization of the periodic orbit of $c$ (of
generation $n$)}. Observe that if the orbit of $c$ is infinite then
this process could be repeated infinitely many times. Otherwise it
can only continue finitely many times and in the end we will get the
rotational renormalization of $f$ of the greatest possible
generation which we will then call the \emph{final rotational
renormalization} of $f$. By Theorem~\ref{TH AlphaFixPt} for
rotational renormalizations of the periodic orbit of $c_f$ of any
generation, including the final renormalization of $f$, the critical
point $c$ is not a fixed point.

We now relate the above to  combinatorial one-dimensional dynamics.
Let $X$ be a dendrite and $P\subset X$ be finite. Suppose that $p\in
P$ and there is a map $f$ defined on $P$ (and maybe elsewhere too)
such that $P=\orb_f(p)$. Consider the continuum hull $T(P)=T$ of
$P$; clearly, $T$ is a tree. Consider two triples $(f, P,X)$ and
$(f', P',X')$ as described above with $f':P'\to P'$ a transitive map
and $P'$ contained in a dendrite $X'$. Suppose that there exists a
homeomorphism $h:T(P)\to T(P')$ which respects the dynamics of
$f|_P$ and $f'|_{P'}$. Then we declare $(f, P,X)$ and $(f', P',X')$
to be equivalent. The class of equivalence of $(f, P,X)$ is called a
\emph{pattern}. If a map $F:X\to X$ of a dendrite and an
$F$-periodic point $x$ are given then we call the pattern of $(F,
\orb_F(x),X)$ the \emph{pattern of $x$}, and we can also say that
$x$ \emph{exhibits} a certain pattern. Lemma~\ref{le-version}
excludes certain types of patterns from the list of possibilities
for periodic orbits of critical points of maps $f\in \tp$. To
describe them we need a few notions.

Suppose that for $(f, P, X)$ there is a partition of $P$ into
cyclically permuted (by $f$) non-degenerate subsets with pairwise
disjoint continuum hulls. Then we call the subsets \emph{blocks} and
say that the pattern has a \emph{block structure} (a block structure
is not unique). Suppose that all points of $P$ are endpoints of
$T(P)$ and there is a point $a\in  T(P)$ such that arcs from $P$ to
$a$ meet only at $a$. Then we can  visualize the action of $f$ on
$P$ as the ``rotation'' of $P$ about $a$. In this case we call the
pattern of $(f, P,X)$ \emph{basic rotational} or a \emph{snowflake
(of generation $1$)}. Similarly, suppose that a pattern of $(f,
P,X)$ has a block structure such that the set-theoretic difference
$B$ between $T(P)$ and the union of all the blocks is connected,
continuum hulls of different blocks are disjoint, and there is a
point  $a$ in $B$ such that all components of $T(P)\setminus \{a\}$
containing different blocks of $P$ are pairwise disjoint (this time
$f$ ``rotates'' the blocks about $a$). Then we say that $(f, P,X)$
exhibits a \emph{non-trivial rotational pattern (of generation
$1$)}. Recall that blocks are non-degenerate by definition. Also, it
is clear that there exist patterns which are neither snowflakes of
generation $1$ nor non-trivial rotational patterns of generation
$1$. However, we are not interested in such patterns and do not
consider them here.

Let $(f,P ,X)$ exhibit a non-trivial rotational pattern with $n_1=n$
blocks $P^1_0,$ $f(P^1_0),$ $\dots,$ $f^{n-1}(P^1_0)$. Consider a
few cases. First, it may happen that $(f^n, f^i(P^1_0),X)$ exhibits
a basic rotational pattern for all $0\le i\le n-1$. In this case we
say that the pattern of $(f, P,X)$ is a \emph{snowflake (of
generation $2$)}. Second, it may happen that for all $i, 0\le i\le
n-1$ the pattern of $(f^n, f^i(P^1_0),X)$ is a non-trivial
rotational pattern with the blocks of $(f^n, f^{i+1}(P^1_0),X)$
being $f$-images of blocks of $(f^n, f^i(P^1_0),X)$. Then say that
the pattern of $(f, P,X)$ is a \emph{non-trivial rotational pattern
of generation $2$}. There exist non-trivial rotational patterns of
generation $1$ which belong to neither of the above classes but we
do not consider them here.

This process can be continued. If a pattern of $(f, P,X)$ is
non-trivial rotational of generation $k$ then there is a block
$P^k_0$ containing $p$ and there are say $n_k$ blocks into which $P$
is partitioned. If now all patterns of $(f^{n_k},$ $f^i(P^k_0),$
$0\le i\le n_k-1$ are snowflakes then we say that the pattern of
$(f, P, X)$ is a \emph{snowflake (of generation $k+1$)}. On the
other hand, if $(f^{n_k}, P^k_0,X)$ exhibits a non-trivial
rotational pattern with the block $P^{k+1}_0$ containing $p$ so that
in fact for any $i, 0\le i\le n-1$, the pattern of $(f^{n_k},
f^i(P^k_0),X)$ is also a non-trivial rotational pattern whose blocks
are the appropriate images of the blocks of $(f^{n_k}, P^k_0)$, then
we say that the pattern of $(f, P)$ is a \emph{non-trivial
rotational pattern of generation $k+1$}. There exist non-trivial
rotational patterns of generation $k$ which belong to neither of the
above classes but we do not consider them. A pattern is called a
\emph{snowflake} if it is a snowflake of some generation.

\begin{lem}\label{le-version} Suppose that $f\in \tp$. Then the pattern of the
periodic orbit of the critical point $c_f$ cannot be a snowflake.
\end{lem}

\begin{proof} Let $f:X\to X$. Suppose first that the pattern of $c_f=c$
is a snowflake of generation 1. Denote the $f$-orbit of $c$ by $P$.
Then there is a point $v\in T(P)$ such that $v=a$, the fixed
cutpoint which belongs to $(c, f(c))$. It follows that  the period
of $c_f$ equals the local period of $f$ at $a$, $f$ is rotationally
renormalizable, and for the first rotational renormalization $F_1$
of $f$ we have that $c_f$ is fixed. However this is impossible by
Theorem~\ref{TH AlphaFixPt}, hence $(f, P,X)$ cannot be a snowflake
of generation 1.

Suppose that $P$ is a non-trivial rotational pattern and show that
then the fixed cutpoint $a$ in $(c,f(c))$ does not belong to the
continuum hull of any block of this pattern. Indeed, suppose
otherwise. Then $a\in T(P_1)$ where $P_1$ is one of the blocks. If
$c\nin P_1$ then by the definition of blocks $c\nin T(P_1)$. If
$c\nin T(P_1)$, then $T(f(P_1))=f(T(P_1))$ since $F|_{T(P_1)}$ is a
homeomorphism. Hence, the fact that $a\in T(P_1)$ implies that
$f(a)=a\in T(f(P_1))$, a contradiction, since by definition
continuum hulls of blocks must be disjoint. Suppose now that $c\in
P_1$. If $a\in T(P_1)$ then there is another point $y\in P_1$ such
that $a\in (c, y)$. Since $f|_{[c, y]}$ is one-to-one, then $a\in
f([c, y])=[f(c), f(y)]\subset T(f(P_1))$, and so again we have
contradiction with the property that continuum hulls of blocks must
be disjoint. Hence $a$ does not belong to the continuum hull of any
block of $(f, P,X)$ and the action of $f$ on $P$ can be viewed as
the ``rotation'' of blocks of $P$ about $a$.

Suppose there are $m$ such blocks. Let us show that then $f$ is
rotationally renormalizable. Indeed, we need to show that the
$f^m$-orbit of $c$ is contained in $R$, the component of $X\setminus
[a,a']$ containing $c$. Observe that similarly to the previous
paragraph we can show that $a'$ does not belong to the continuum
hull of any block of $P$. On the other hand, there is only one block
of $P$ contained in $K$ (recall that $K$ is the component of
$J\setminus \{a\}$), namely the block $Q$ to which $c$ belongs.
Indeed, otherwise $P$ would not be a non-trivial rotational pattern
of generation 1, a contradiction with the assumption. Thus, the
entire $f^m$-orbit of $c$ coincides with $Q$, and since $a'\nin
T(Q)$ then $Q\subset R$ as desired. Thus, $f$ is rotationally
renormalizable and we can continue the same arguments now applying
them to $f^m|_{R_\iy}$. Repeating the construction, we see that if
the orbit of $c$ is a snowflake then eventually we will get a
renormalization of $f$ for which the critical point will be fixed, a
contradiction with Theorem~\ref{TH AlphaFixPt}. This completes the
proof of the lemma.
\end{proof}

Snowflakes have already been studied in a different context. Namely, in
\cite{blo91} continuous tree maps were considered and patterns of zero entropy
tree maps fully described. It turns out that a continuous zero entropy tree map
can only have periodic points whose patterns are snowflakes (which explains the
title of the paper \cite{blo91}). The reason they appear here as well is that
for the tree dynamics the patterns of periodic orbits not forcing positive
entropy and the patterns forcing (in the absence of critical points outside the
periodic orbit) the existence of either identity return or attracting periodic
point are the same.


\section{Renormalization of Laminations}\label{relam}

In Section~\ref{relam} we define {\em renormalizations of laminations}
in parallel to the renormalization on the dendrites (Section~\ref{per})
and solve the main problem.
Throughout the section we assume that a lamination
with periodic critical leaf $\ell_\ta=\ol{\ta\ta'}$ is given (i.e., $\si(\ta)$ is periodic).

We need a few notions. The family of orientation preserving
\emph{homeomorphisms} $h:\ucirc \to \ucirc$ is denoted by $\Ho$; the family of
orientation preserving monotone \emph{maps} $\ucirc\to \ucirc$ is denoted by
$\M$. Suppose $A, B\subset \ucirc$, and $f:A\to A, g:B\to B$ are two maps.  We
call $f$ and $g$ \emph{conjugate (monotonically semiconjugate)} if there is a
map $h\in \Ho$ ($m\in \M$) which conjugates (semiconjugates) $f|_A$ to $g|_B$.
A closed $\si$-invariant set $D\subset \ucirc$
is said to be \emph{rotational (of rotation number $0\le \rho<1$)} if:

\begin{enumerate}

\item $D$ is a periodic orbit on which $\si|_D$ is  conjugate to the restriction
of the rigid rotation by the rotation angle $\rho$ on the orbit of $0$, \,
\textbf{or}

\item $\si|_D$ is monotonically semiconjugate to the irrational rigid rotation
by the angle $\rho$.

\end{enumerate}

If $A, B$ are finite and $f|_A$ and $g|_B$ are conjugate, say that
$A$ and $B$ \emph{exhibit the same pattern} (then the combinatorics
of $g|_P$ and $h|_Q$ are the same). In particular, if $g$ is a rational rotation then
$A$ is said to be a \emph{rotational periodic orbit}.
If $f|_A$ is monotonically
semiconjugate to $g|_B$, say that $f|_A$ (or just $A$) has a
\emph{block structure} over $g|_B$ (or just $B$). Then there are several pairwise
disjoint arcs in $\ucirc$ containing \emph{blocks} of $A$, and in addition
$f|_A$ and $g|_B$ are 1-to-1 (e.g., if both $A, B$ are periodic orbits) then
blocks of $A$ are mapped onto blocks of $A$ in the same order as points of $B$ are mapped.

Let us discuss the results of Section~\ref{relam}.
In Subsection~\ref{sec-bascas} we consider two basic cases.
If an endpoint of $\ell_\ta$ has an appropriate rotational orbit, we prove
in Lemma~\ref{TH RotnlCase} that non-degenerate laminations $\sim$
compatible with $\ell_\ta$ \emph{do not exist}. However, if the opposite extreme takes
place and the periodic orbit in question does not even have a block structure over
the appropriate rational rotation then (Theorem~\ref{th-basnonrot})
non-degenerate laminations $\sim$ compatible with $\ell_\ta$ \emph{do exist}.

These two extreme cases are like two outcomes of a verification test
of whether non-degenerate laminations $\sim$ compatible with $\ell_\ta$ exist.
There is however a third possible outcome: the test is inconclusive, the
periodic orbit in question has a non-trivial block structure over the appropriate rational
rotation. This case is considered in
Subsection~\ref{renorm}.  There we introduce a version of renormalization for
invariant laminations and apply the obtained results to solve the Main Problem
and give a combinatorial criterion for the existence of a non-degenerate
lamination $\sim$ compatible with $\ell_\ta$.
Given $\ga$, set $\ga'=\ga+\frac 12$.

\begin{thm}[\cite{BS}]\label{lm-bs} Let $\ta\in [0, \frac 12)$. The semi-circle
$[\ta, \ta']$ contains a unique minimal rotational set
$A_{\ta}$ of {\em rotation number $\rho_\ta=\rho\in [0, 1)$}.
If $\rho$  is irrational, then $A_\ta$ is
a Cantor set on which $\si$ is semiconjugate to the irrational rotation by
$\rho$ and $\ta, \ta'$ belong to $A_{\ta}$ (so, if $\ta$ is preperiodic, $A_{\ta}$ is a periodic
orbit). If $\rho$ is rational, $A$ is a unique rotational periodic orbit of rotation number $\rho$.
The unique invariant set in $[\ta', \ta]$ is $\{0\}$.
\end{thm}

\subsection{Basic rotational and non-rotational cases}\label{sec-bascas}

\subsubsection{Basic Rotational Case}
Assume that $\ta\in [0, \frac 12)$. A critical leaf $\ell_\ta$ and the angles $\ta, \ta'$ are said to be
\emph{basic rotational} if $\{\ta, \ta'\}\cap A_\ta\not=\0$ and $\si(\ta)$ is
periodic.  Then $A_\theta=\orb(\theta)$ or $A_\theta=\orb(\theta')$.
In Theorem~\ref{TH RotnlCase} we solve the main problem for basic rotational
critical leaves.

\begin{thm}\label{TH RotnlCase}
Let $\ta \in [0, \frac{1}{2})$ and $\si(\ta)$ be periodic. Let $\sim$ be a
non-degenerate
invariant lamination and suppose that $\ta\sim \ta'$. Then:

\begin{enumerate}

\item if $\al, \be$ are such that for any $k$ the
angles $\si^k(\al), \si^k(\be)$ belong either to $[\ta, \ta']$
or to $[\ta', \ta]$ then $\al\sim \be$;

\item the geometric lamination $\lam^\ta_\iy$ is compatible with $\sim$;

\item the periodic orbit $A_\ta$ and $\{0\}$ are two invariant $\sim$-classes.

\end{enumerate}

In particular, if $\ta$ is basic rotational then a non-degenerate lamination $\sim$
with $\ta\sim \ta'$ does not exist.

\end{thm}

\begin{proof} (1) Let $J_\sim=\ucirc/\sim$, by Lemma~\ref{TH PerPropts} the
topological Julia set $J_\sim$ is a dendrite. Under the conditions from the theorem consider the branched covering
map $\hf_\sim:\sphere \to \sphere$
defined on the \emph{entire} sphere (see Section~\ref{invlam} for the
description of $\hf_\sim$). The topological external rays $R_\ta$ and
$R_{\ta'}$ corresponding to the angles $\ta$ and $\ta'$ land on the same point in $J_\sim$ and  divide $\sphere$ into
two halves whose closures will be denoted $A$ and $B$. Denote the landing
points of the topological rays $R_\al, R_\be$ (corresponding to angles $\al,
\be$) by $z_\al, z_\be$ respectively. It follows that the two topological
external rays $R_\al$ and $R_\be$ corresponding to $\al$ and $\be$ are such
that for any $n$ we have both $\hf^n_\sim(R_\al)\cup \hf^n_\sim(R_\be)$ contained either
in $A$ or in $B$. Assume that $z_\al\ne z_\be$, then $f_\sim^n$ maps $[z_\al, z_\be]$
homeomorphically onto its image for any $n$ which is impossible by
Theorem~\ref{TH AlphaFixPt}, a contradiction.

(2) By the construction, if $\ol{ab}$ is a leaf of $\lam^\ta_\iy$ then for any $n$
$\ol{\si^n(a)\si^n(b)}$ does not cross $\ol{\ta\ta'}$. Hence by (1) $a\sim b$ as desired.

(3) Since all angles from $A_\ta$ have orbits
contained in $[\ta, \ta']\subset \ucirc$, they all are $\sim$-equivalent
and $A_\ta$ is contained in a $\sim$-class $g$. If there exists $x\in
g\setminus A_\ta$, then by Theorem~\ref{lm-bs} there exists $m$ with $\si^m(x)\in
[\ta', \ta]$. Hence a chord connecting $\si^m(x)$ and any point of $A$ intersects
$\ell_\ta$ implying that in fact $g$ is a critical class. However $g$ contains
an invariant set $A$, hence $g$ is an invariant class too. If $p:\ucirc\to
J_\sim$ is the quotient map then it follows from Lemma~\ref{TH PerPropts} that
$f_\sim:J_\sim\to J_\sim$ has a fixed critical point, a contradiction to
Theorem~\ref{TH AlphaFixPt}. Hence $A=g$ is an invariant $\sim$-class.
Similarly one shows that $\{0\}$ is an invariant $\sim$-class. Finally, suppose
that $\ta$ is basic rotational. Then $A_\ta$ is non-disjoint from $\{\ta,
\ta'\}$, hence if $\sim$ exists, it has to contain both points $\ta, \ta'$ and
cannot coincide with $A_\ta$, a contradiction. Hence, a non-degenerate lamination $\sim$ does not exist in
this case.

\end{proof}

\subsubsection{Basic Non-Rotational Case} Assume that $\ta\in[0,1/2)$ and let
$\ell_\ta$ be a critical leaf joining the points
 $\ta$ and $\ta'$. We now consider the case
that  $\si(\ta)$ is periodic and $A_\ta\cap\{\ta,\ta'\}=\0$. In this two possibilities are possible.
If $\ell_\ta$ is basic non-rotational (defined below), then
 we prove in Theorem~\ref{th-basnonrot} that in this
case a non-degenerate lamination $\sim$ with
$\ta\sim \ta'$ exists. As a tool we develop the idea of a {\em traveling horseshoe}
$D_\iy(A_\ta)$.

For a periodic orbit $A$, to define $D_\iy(A)$ we need a few notions which can
be developed more generally without reference to any critical leaf. Let
$I=[\al, \be], I'=[\al', \be']\subset \ucirc$ be two disjoint closed arcs not
containing $0$ or $\frac 12$, with either $\be, \al'$ being $k$-periodic, or
$\be', \al$ being $k$-periodic. Suppose for the sake of definiteness that $\be,
\al'$ are $k$-periodic and that $\al<\be<\al'<\be'$. If we move along the
circle from $\al'$ to $\be'$ then the $\si^k$-image of our point moves from
$\al'$ to $\be$. In the simplest case $\si^k$ maps $I$ onto $[\al', \be]$
homeomorphically, but it may happen that the $\si^k$-image wraps around the
circle a few times. In any case eventually it comes to $\si^k(\be')=\be$.
Similarly the $\si^k$-image of $I$ can be described.

Choose four intervals inside $I\cup I'=D$ as follows: (1) choose the interval
from $\be$ to the $\si^k$-preimage of $\al$ closest to $\be$ inside $I$ and
denote this interval $I_{00}$; (2) choose the interval from $\al$ to the
$\si^k$-preimage of $\be'$ closest to $\al$ and denote this interval $I_{01}$;
(3) similarly choose intervals $I'_{11}, I'_{10}\subset [\be', \al']$; (3) set
$H(I, I')=H=I_{00}\cup I_{01}\cup I'_{10}\cup I'_{11}$. Clearly, $0,
\frac12\nin \bigcup^k_{i=0} \si^i(H)$. Define the set $D_\iy(I, I')=D_\iy$ as
the set of all points which stay inside $H$ under $\si^k$. It follows that
$D_\iy$ is a Cantor set on which $\si^k$ is conjugate to the one-sided 2-shift;
$D_\iy$ is called a \emph{general horseshoe (of period $k$)}. The open arcs in
$\ucirc$ complementary to $D_\iy$ are said to be \emph{holes (in $D_\iy$)}.

We define a map $\ph$ which collapses all closed holes to points. It follows
that $\ph(\ucirc)=\ucirc$ and $\si^k|_{D_\iy}$ is semi-conjugate by $\ph$ to
$\si|_\ucirc$. Call the arcs $(\be', \al)$ and $(\be, \al')$ (i.e., the arcs
complementary to $I\cup I'$) the \emph{main holes (in $D_\iy$)} with their
union denoted by $M(D)$ (we will use this notation later). Also,
there are two holes whose endpoints map onto the non-periodic
endpoints of one of the main holes (the arcs themselves wrap around the circle
one or more times). These two holes are said to be \emph{premain}. All
other holes are said to be \emph{secondary}. Set the $\ph$-images of main
closed holes $[\be', \al]$ and $[\be, \al']$ to be points $\frac 12$ and $0$,
respectively.  Set the $\ph$-images of premain holes to be points $\frac 14$
(for the hole contained in $[\al, \be']$) and $\frac 34$ (for the hole
contained in $[\al, \be']$. Inductively, map secondary holes to
appropriate diadic rational angles. Thus, $\ph(\be)=\ph(\al')=0,
\ph(\al)=\ph(\be')=\frac 12$ etc. The map can then be extended onto the entire
circle; it collapses all holes in $D_\iy$ and maps $D_\iy$ onto $\ucirc$. We
call $\ph$ the \emph{pruning of $\ucirc$ by $D=I\cup I'$}; also, in this
setting we call $\ucirc$, understood as the $\ph$-image of $D_\iy$, the
\emph{$\ph$-circle}. The pruning $\ph$ monotonically semiconjugates
$\si^k|_{D_\iy}$ and $\si|_{\ucirc}$.

We now need to outline our plan of solving the main problem. First let us become
more specific with respect to the behavior of $D_\iy$ under other powers of
$\si$. Suppose that in the situation above the convex hulls of sets
$\si(H),\dots, \si^{k-1}(H)$ are disjoint from the convex hull of $H$. Then we
say that $D_\iy$ is a \emph{traveling horseshoe (of period $k$)}. E.g., a
traveling horseshoe can be generated by two intervals $I$ and $I'$ as above if
$\si(I)=\si(I'), \dots, \si^{k-1}(I)$ are disjoint from $I\cup I'=D$.
Now, given a periodic critical leaf $\ell_\ta$, by Theorem~\ref{lm-bs}
the semicircle $[\ta,\ta']$ contains a minimal rotational periodic orbit $A_\ta=A$
of period, say, $k$. Then we construct below a \emph{canonic traveling horseshoe
(associated to $A$)} (and hence to $\ell_\ta$). This horseshoe
$D_\iy(\ta)=D_\iy$ travels so that the orbit $Z(\ta)=Z$ of $D_\iy$ has block
structure over $A$ (hence all invariant subsets
of $Z$ have block structure over $A$). We show that the opposite is
also true (for brevity we do this only for periodic orbits but  the claim holds for any set):
if a periodic orbit $P$ has a block structure over $A$ then it is contained in the orbit of $D_\iy$.
Hence the pattern of a periodic orbit already shows if the orbit is contained in $Z$ or not.

Let us explain how we use these tools to solve the main problem; denote the orbit of $\si(\ta)$ by $P$.
We show that the \emph{canonic pruning $\ph$}
at most 2-to-1 semiconjugates $\si^k|_{D_\iy}$ to $\si$. We use this to show that if $\ta\nin D_\iy$
(equivalently, the orbit of $\si(\ta)$ does not
have a block structure over $A$) then the lamination $\sim_{\lam^{\ta}_{\iy}}$ constructed
in Section~\ref{geola} is non-degenerate. Now, suppose that $\ta\in D_\iy$ (and hence $P\subset Z$).
Then we transport the block of
$P$ contained in $D_\iy$ to a periodic orbit $Q$ of $\si|_{\ph(D_\iy}$. Then $Q$ is called
a \emph{rotational renormalization} of $P$. The construction is correct because
for any critical diameter $\ell_\ga$ such that $\ga$ comes from the pair of intervals generating
$D_\iy$ we have $A(\ga)=A(\ta)$. Now we apply the same arguments to $Q$
and proceed similarly which in the end leads to the main result of the paper.

Let us now make a few remarks.
In the case when $\si(\ta)$ is periodic, there is a possible source of confusion: the periodic
orbit of $\si(\ta)$ can be a rotational orbit, but not be $A_\ta$. This is
because $A_\ta$ and the orbit of $\si(\ta)$ are the same periodic orbit iff
$A_\ta$ is entirely on one side of the diameter $\ol{\ta\ta'}$.  For example,
if $\ta$ (or $\ta'$) happens to be in a rotational orbit $B$, but points of $B$
are on both sides of $\ol{\ta\ta'}$ (i.e., if $\ta$ - or $\ta'$ - is not an
endpoint of the complementary to $B$ arc containing $0$), then $B$ is the
rotational orbit associated with a different diameter than $\ol{\ta\ta'}$.

The length of an arc $I$ is denoted below by $|I|$. Let $\ell_\ta$ be a periodic
critical leaf. By Theorem~\ref{lm-bs} it gives rise to the rotational  periodic orbit
$A(\ta)=A$ of some rational rotation number $\rho=\frac mk \in
\ucirc$ (we \emph{always} consider rational numbers in the
reduced form); moreover, $A(\ta)$ is then the unique
rotational periodic orbit with this rotation number. Below we introduce some objects depending on $A$,
however this dependence is omitted for the time being (later we reflect this
dependence in our notation). These objects can also be viewed as depending on $\ta$.

Let the components of $\ucirc \setminus A$ be $I_1, \dots, I_k$ with
$I_1<I_2<\dots<I_{k-1}<\frac 12<I_k$ (by Theorem~\ref{lm-bs} this is correct).
Then $\si^{k-1}|_{I_1}$ is a homeomorphism onto $I_k$. Following Milnor
\cite{mi}, set $I_1 = (2\al, 2\be)$ (then $2\al, 2\be \in A$) and
$\si^{-1}(\ol{I_1}) = D$. Then $D \subset \ol{I_k}$ is the disjoint union of
two arcs $I = [\al, \be]$ and $I' = [\al', \be']$, each of which maps by $\si$
onto $\ol{I_1}$; the restriction of $\si$ on either $I$ or $I'$ is an expanding
homeomorphism. Also, $\be, \al' \in A$ are $\si^k$-fixed. The map $\si^k$ maps
both $I$ and $I'$ onto $\ol{I_k}$ homeomorphically and expands the length by
the factor of $2^k$. So, $D$ generates a traveling horseshoe $D_\iy=D_\iy(A)$
of period $k$ called the \emph{canonic traveling horseshoe}, or just
\emph{horseshoe (associated to $A$)}. The corresponding \emph{canonic pruning}
was considered in \cite[Chapter 2]{BS}. A leaf $\ell_\ta$ is said to be
\emph{basic non-rotational} if $\{\ta, \ta'\}$ is disjoint from $D_\iy(A_\ta)$.

After the first application of $\si$ which maps $D_\iy$ into $[2\al, 2\be]=I_1$
the set $D_\iy$ is ``traveling'' in $\ucirc$ together with $I_1$ following
the pattern of $A$ until $\si^{k-1}$ maps $I_1$ onto $I_k$ and $D_\iy$ onto
itself. As above, let $Z(\ta)=Z$ be the orbit of $D_\iy$. We now prove Lemma~\ref{orbind}
which relates $Z$ and orbits having block structure over $A$. It allows us to see
if $P$ is contained in $Z$ from the pattern of $P$ alone.

\begin{lem}\label{orbind} A periodic orbit $P$ has block structure over $A$ iff
$P\subset Z$.
\end{lem}

\begin{proof} Clearly, if $P\subset Z$ then it has block structure over $A$. Suppose
now that $P$ has block structure over $A$. Then given a block $H$ there are well-defined
points $\al(H)=\al, \be(H)=\be\in H$ so that $H\subset [\al, \be]$. Let us call $[\al, \be]$
the \emph{span (of $H$)} ad denote it $\sn(H)$. By the definition spans of blocks are
disjoint, and in particular
$\si(\al)\nin [\al, \be], \si(\be)\nin [\al, \be]$. This easily implies that if $0\nin [\al, \be]$ then
$[\al, \be]$ and $[\si(\al), \si(\be)]=\si([\al, \be])$ are disjoint. Hence $\al(\si(H))=\si(\al),
\be(\si(H))=\si(\be)$. Thus, for all blocks whose spans do not contain $0$ the map $\si$ does not
change the relative order of points in the block and expands the length of the span twofold.
Thus, exactly one span contains $0=x_k$.

Denote the spans $H_1, \dots, H_k$ so that $0\in H_k$ and $\si(H_j)=H_{j+1}, 1\le j<k$.
Then $\frac 12=x_{k-1}\in H_{k-1}$. Recall, that $A$ divides $\ucirc$ into
arcs $I_1, \dots, I_k$ introduced above; these arcs are analogous to spans and are ordered
on the circle the same way. Then $\frac 12\in I_{k-1}$. Now, let us denote the further preimages
of $0$ inside $H_{k-2}, \dots, H_1$ by $x_{k-2}, \dots, x_1$. Let us also denote the further preimages
of $0$ inside $I_{k-2}, \dots, I_1$ by $y_{k-2}, \dots, y_1$. The points $\{x_1, \dots, x_{k-1}, x_k\}$ and the
points $\{y_1, \dots, y_{k-1}, y_k\}$ are ordered on the circle the same way which coincides with the
circular order of points in the rotational periodic orbit $A$.
Let us show that then $y_j=x_j, j=1, \dots, k-2$. Indeed, the point $x_{k-2}$ is located with respect to the points
$0, \frac 12$ exactly where the order of points dictates, the same applies to $y_{k-2}$, and since this is the same order
then $y_{k-2}=x_{k-2}$. The same argument shows that $y_j=x_j, j=1, \dots, k$.

Let us show that $H_1\subset I_1=[u, v]$. Clearly, $H_1$ covers $0$ for the first time when it maps
(1-to-1) by $\si^{k-1}$ onto $H_k$ and that $\frac 12\nin H_k$ (because $k>1$). So, there is
only one $\si^k$-preimage of $0$ in $H_1$ (coinciding with $\si^{k-1}$-preimage of $0$ in $H_1$).
Suppose that $H_1\not\subset I_1$. We may assume that there is an interval $I_j$ adjacent to $I_1$
(say, their common endpoint is $u$) such that $H_1\cap \inte(I_j)\ne \0$.
Clearly, $H_1$ cannot contain $I_j$
because otherwise there is an earlier than $\si^{k-1}(H_1)=H_k$ image of $H_1$ containing $0$, a contradiction.
Hence we may assume that $\al(H_1)=\al_1\in \inte(I_j)$. Since $\si^k(\al_1)$ must belong to $H_1$ we see that
$\si^k$-image of $[u, \al_1]$ stretches over $0$ and contains yet another, different from $x_1\in I_1$,
$\si^k$-preimage of $0$, a contradiction. Hence $H_1\subset I_1$. This implies that $\si^{k-1}(H_1)=H_k\subset I_k$,
so the points of $P\cap H_k$ belong to the set $D_\iy$ of points which map by $\si^k$ back to $H_k$
and $P\subset Z$ as desired.
\end{proof}

So, $P$ does not have block structure over $A$ iff $\ta\nin D_\iy$. Then
we call $\ell_\ta$ \emph{basic non-rotational}; Theorem~\ref{th-basnonrot} solves the Main Problem for
such leaves. Recall that $M(D)$ denotes the union of main holes of $D$.

\begin{thm}\label{th-basnonrot} Suppose that $D=I\cup I'$ generates a
traveling horseshoe $D_\iy$ of period $k$. Let $\ell_\ta$ be a
critical leaf such that $\ta, \ta'\nin D_\iy\cup M(D)$ (e.g.,
$\ell_\ta$ may be a basic non-rotational critical leaf, i.e. such that the
rotational set $A_\ta$ is a periodic orbit of period $k$ and $\ta, \ta'\nin D_\iy(\ta)$).
Then there exists a non-degenerate lamination $\sim$ with $\ta\sim \ta'$. In particular,
for a basic non-rotational critical leaf there is always a compatible non-degenerate
lamination.
\end{thm}

\begin{proof} The remark in parentheses in the statement of the theorem
is justified by the explanations before the theorem where we show that
in the basic non-rotational case $D=I\cup I'$ generates a
traveling horseshoe $D_\iy$ of period $k$. Also, it follows from the definition
of $A_\ta$ that in that case $\ta, \ta'\nin M(D)$, and hence $\ta, \ta'\nin D_\iy\cup M(D)$.

As a non-degenerate lamination $\sim$ with $\ta\sim \ta'$ we choose the
lamination constructed as follows: (1) we construct the geometric lamination
$\lam^\ta_\iy$ as in Section~\ref{geola}; (2) then we construct the lamination
$\sim=\sim_{\lam^\ta_\iy}$ as in Theorem~\ref{omega closed} and show that
$\sim$ is not degenerate. We use the notation from above, in particular $\ph$
is the pruning by $D$; also, we use notation like
$\hat{\ell}, \widehat{C}$ etc for leaves and classes in the $\ph$-circle.

Clearly, $\ta\in U$ where $U$ is a non-main hole in $D_\iy$, and $\ta'\in U'$.
For some $q\ge 0$ both $U$ and $U'$ map by $\si^{kq}$ onto two
premain holes and then by $\si^k$ onto the main hole with non-periodic endpoints.
Recall that by the construction this main hole maps by $\ph$ to
$\frac 12$. Hence, both points $\ph(\ta)$ and $\ph(\ta')=(\ph(\ta))'$ are
\emph{$\si$-preimages (under some power)} of $\frac 12$ and by
Theorem~\ref{thnonper} the geometric lamination $\widehat{\lam}^{\ph(\ta)}_\iy$
generates a non-degenerate lamination $\approx$ in the $\ph$-circle. By
Theorem~\ref{TH RotnlCase} the $\approx$-class $\widehat{B}$ of $0$ is $\{0\}$,
thus $\widehat{C}=\{\ph(\ta), \ph(\ta')\}$ is an $\approx$-class, and hence by
Lemma~\ref{REM *} the class $\widehat{C}$ is the unique critical
$\approx$-class. Moreover, all leaves of the geometric lamination
$\lam_\approx$ associated with $\approx$ are approximated from at least one
side by pullbacks of $\hat{\ell}_{\ph(\ta)}$. Indeed, $J_\sim$ is a dendrite
by Lemma~\ref{TH PerPropts}. Then by Theorem~\ref{TH AlphaFixPt} precritical
points are condense in $J_\approx$ which, in the language of laminations
translates exactly into the above statement (observe that since $\approx$ is a
lamination no two gaps of $\lam_\approx$ can meet over a leaf). Together
with the construction of the pullback lamination $\widehat{\lam}^{\ph(\ta)}_\iy$
this implies that in fact $\lam_\approx=\widehat{\lam}^{\ph(\ta)}_\iy$.

Consider in the $\ph$-circle a leaf $\widehat{\ol{\al\be}}$ which is a $\si^t$-preimage
leaf of $\hat{\ell}_{\ph(\ta)}$. The fundamental construction of the pullback
lamination (see Section~\ref{geola}) shows that the forward images of
$\widehat{\ol{\al, \be}}$ are disjoint from $\hat{\ell}_{\ph(\ta)}$ except for
the last one which coincides with $\hat{\ell}_{\ph(\ta)}$. Then
$\ph^{-1}(\al)=W$ and $\ph^{-1}(\be)=V$ are two holes in $D_\iy$ and the convex
hulls $\ch(W, V),$ $\ch(\si^k(W), \si^k(V)),$ $\dots, \ch(\si^{(t-1)k}(W),
\si^{(t-1)k}(V))$ are disjoint from $\ch(U, U')$ (because these convex hulls
are $\ph$-preimages of pairwise disjoint $\approx$-leaves). Hence we can define
the $\si^{tk}$-pullback $\ell$ of $\ell_\ta$ whose endpoints belong to $W$ and
$V$, and whose $\si^k$-images have the endpoints in the appropriate
$\si^k$-images of $W$ and $V$ respectively (and hence are disjoint from
$\ell_\ta$) until finally $\si^{tk}(\ell)=\ell_\ta$.

To show that other $\si$-images of $\ell$ are disjoint from
$\ell_\ta$ observe that by the construction the convex hulls of
$D_\iy$ and its images contain $\ell$ and its images. On the other hand, by the
definition the convex hull of $D_\iy$ and the convex hulls of $\si$-images of
$D_\iy$ are contained in $I_1, \dots, I_k$ and hence are pairwise disjoint
(except for the boundaries). We conclude that the forward orbit of $\ell$ is
disjoint from $\ell_\ta$ except for the last image $\si^{tk}(\ell)=\ell_\ta$.
Hence all such leaves $\ell$ belong to $\lam^\ta_\iy$ (see
Section~\ref{geola}). Thus for all preimage leaves of finite pullback
laminations $\widehat{\lam}^{\ph(\ta)}_n$ there are corresponding finite
preimage leaves of $\ell_\ta$ in $\lam^\ta_\iy$ ordered in the disk the same way. Since
$\lam_\approx=\widehat{\lam}^{\ph(\ta)}_\iy$ it implies that to all leaves of
$\widehat{\lam}^{\ph(\ta)}_\iy$ there are associated leaves of $\lam^\ta_\iy$.
We conclude that there are uncountably many pairwise disjoint leaves of
$\lam^\ta_\iy$; denote the collection of these leaves by $\A$.

Now it follows from Theorem~\ref{omega closed} similarly to the proof of
Theorem~\ref{thnonper} that the lamination $\sim_{\lam^\ta_\iy}$ is
non-degenerate. Indeed, set $\lam=\lam^\ta_\iy$ and let $\Lstar$ be the union
of all leaves of $\lam$. Choose two points $z, \zeta\in \ucirc$ so that an
uncountable family of leaves of $\A$ separates them. If a continuum $K\subset
\Lstar$ contains $z, \zeta$ then it has to cross every leaf of $\A$, hence by
Lemma~\ref{contain} it has to cross $\ucirc$ over an uncountable set of points
and cannot be an $\om$-continuum. Thus, $z\not\sim_\lam \zeta$ and hence
$\sim_\lam$ is not degenerate as desired.
\end{proof}

\subsection{Renormalization}\label{renorm}

The case not yet covered by the two basic cases is that when for a
periodic critical leaf $\ell_\ta$ we have $\{\ta, \ta'\}\subset
D_\iy(A_\ta)\setminus A_\ta=D_\iy\setminus A$ (we assume for definiteness that
$0<\ta<\frac 12$ and $A$ is of period $k$). To consider this case we first
assume that a non-degenerate lamination $\sim$ compatible with $\ell_\ta$
exists and draw appropriate conclusions which are necessary conditions on
$\ta$ for the existence of a lamination compatible with $\ell_\ta$.

The first step here reflects the construction of rotational renormalization on
dendrites from the second half of Section~\ref{per}. To make this step we only
need to assume that $\{\ta, \ta'\}\subset D_\iy$ without assuming the
periodicity of $\ell_\ta$; for simplicity we also assume that $\ta$ is not
mapped into $A$ by powers of $\si$ (this holds if $\ta$ is periodic but not
basic rotational). By Theorem~\ref{TH RotnlCase} we see that if $\{\ta,
\ta'\}\subset D_\iy$ then we can consider the rotational renormalization $F_1$
of the induced map $f=f_\sim$ defined on the dendrite $R_\iy$ (see
Lemma~\ref{le-Rinf}). Then the angles corresponding to the points of $R_\iy$
are exactly the angles of the set $D_\iy$. Say that two angles $\al, \be\in
\ucirc=\ph(D_\iy)$ are \emph{$\sim_1$-equivalent} if there are elements of
$\ph^{-1}(\al), \ph^{-1}(\be)$ which are $\sim$-equivalent where $\ph$ is the
appropriate canonic pruning.

\begin{lem}\label{le-lamrenorm} The relation $\sim_1$ is an invariant lamination
such that $f_{\sim_1}:J_{\sim_1}\to J_{\sim_1}$ and $F_1:R_\iy\to R_\iy$ are
conjugate. Moreover, the critical leaf $\ph(\ell_\ta)$ is compatible with $\sim_1$.
\end{lem}

\begin{proof} We use the notation introduced when we defined the canonic pruning.
Thus, the smallest complementary to $A$ arc is $I_1=(2\al, 2\be)$; we consider
two arcs $D^{-} = [\al, \be]$ and $D^{+} = [\al', \be']$. Each of $D^{-}$ and
$D^{+}$ homeomorphically maps by $\si$ onto $\ol{I_1}$, and then eventually by
$\si^k$ onto $[\al', \be]$ (which gives rise to the set $D_\iy$). Since by
Lemma~\ref{TH RotnlCase} $A$ is a $\sim$-class then $A'$ is a $\sim$-class too.

Let us now show that the endpoints $u, v$ of a hole $(u, v)$ in $D_\iy$ are
$\sim$-equivalent. Since the points $\si^k(u), \si^k(v)$ are the endpoints of
various holes in $D_\iy$ then the chord $\ol{\si^k(u)\si^k(v)}$ never crosses
$\ell_\ta$ inside $\disk$. Thus, by Theorem~\ref{TH RotnlCase} $u\sim v$.
Moreover, a main hole with non-periodic endpoints $(\be', \al)$ is a
homeomorphic image of $(u, v)$. Hence by the properties of laminations the
$\sim$-class of $\{u, v\}$ is the appropriate preimage of $A'$ in $(u, v)$;
only points $u, v$ in this $\sim$-class belong to $D_\iy$. This implies that if
$x\in (u, v)$ does not belong to the $\sim$-class of $\{u, v\}$ then it cannot
belong to a $\sim$-class of a point of $D_\iy$ because otherwise two leaves of
the associated lamination $\lam_\sim$ would cross inside $\disk$. Hence if
$y\in D_\iy$ is not an endpoint of a hole in $D_\iy$ then its $\sim$-class $Y$
is contained in $D_\iy$ completely and consists of points which are not
endpoints of holes in $D_\iy$. Thus $\ph|_Y$ is 1-to-1 which implies that
$\ph(Y)$ is a $\sim_1$-class. Also, if $(u, v)$ is a hole in $D_\iy$ then by
the above $\ph(u)=\ph(v)$ is a $\sim_1$-class. Finally, by the construction the
critical leaf $\ph(\ell_\ta)$ is compatible with $\sim_1$. It follows from the
definitions of both $\sim_1$ and $F_1$ that $f_{\sim_1}:J_{\sim_1}\to
J_{\sim_1}$ and $F_1:R_\iy\to R_\iy$ are conjugate.
\end{proof}

The lamination $\sim_1$ with the critical leaf $\ph(\ell_\ta)$ is called the
\emph{rotational renormalization (of generation $1$)} of $\sim$ which is
defined by a periodic critical leaf $\ell_\ta$. We consider $\sim_1$
analogously to $\sim$ and depending on its dynamics introduce \emph{rotational
renormalizations} of $\sim$ of higher generations denoted by $\sim_2, \sim_3,
\dots$. The process of renormalization applies to the orbit $Q$ of $\{\ta,
\ta'\}$ and to the critical leaf $\ell_\ta$ yields a sequence of their
renormalizations, periodic orbits $Q_1, Q_2, \dots$ and critical leaves
$\ell_1=\ph(\ell_\ta), \ell_2, \dots$. The process stops in two cases. First,
when the renormalization $Q_k$ of $Q$ is basic non-rotational. In this
case we call $\ell_\ta$ (or $Q$) a \emph{critical leaf (or orbit) of rotational
depth $k$}. Second, $Q_k$ can be such that the corresponding critical leaf
$\ell_k$ is basic rotational. Then we say that $\ell_\ta$ \emph{generates} its orbit
$Q$ called a \emph{laminational snowflake of depth $k$}. The connection between
$\si$ and the induced map of the quotient space $\ucirc$ implies that the
rotational renormalizations of the periodic orbits on the circle correspond to
the rotational renormalizations of the orbits of the corresponding periodic
points of the induced map and allows us to prove Theorem~\ref{le-rotsnow}.

\begin{thm}\label{le-rotsnow} Let $\ta \in [0, \frac{1}{2})$ and
$\ell_\ta$ generates a laminational snowflake of some depth. Then a
non-degenerate lamination $\sim$ with $\ta\sim \ta'$ does not exist.
\end{thm}

\begin{proof} Suppose otherwise. Then by Lemma~\ref{le-lamrenorm} we can define
the lamination $\sim_1$, the rotational renormalization of $\sim$ of generation
$1$ as well as the rotational renormalization $F_1$ of the induced map
$f=f_\sim$ defined on the dendrite $R_\iy$ (see Lemma~\ref{le-Rinf}). Moreover,
the critical leaf $\ph(\ell_\ta)$ is compatible with $\sim_1$. Clearly,
$\ph(\ell_\ta)$ is a periodic critical leaf but of less period. Then we will
define the rotational renormalization of $\sim$, now of generation $2$, etc. On
all these steps laminations $\sim_1, \sim_2, \dots$ will not be degenerate and
will correspond to non-degenerate quotient spaces with non-degenerate induced
maps. However, the process of defining the rotational renormalizations of
$\sim$ of higher generations has to stop because the critical leaf $\ell_\ta$
is periodic. By the definition of a critical leaf which generates a
laminational snowflake of some depth, it can only stop when on the next
step the  periodic critical leaf of the next rotational
renormalization of $\sim$ is basic rotational which is impossible by Theorem~\ref{TH
RotnlCase}.
\end{proof}

To consider the remaining case we prove the following theorem.

\begin{thm}\label{th-nonrot} Let $\ell_\ta$ be a periodic rotational critical
leaf of rotational depth $m$. Then there exists a traveling horseshoe which together
with $\ta, \ta'$ satisfies the conditions of Theorem~\ref{th-basnonrot}. In
particular, there exists a non-degenerate lamination $\sim$ with $\ta\sim
\ta'$.
\end{thm}

\begin{proof} We
consider the renormalizations of $\ell_\ta$ in a step by step fashion. They all
will be rotational, until the $m$-th renormalization which will be basic
non-rotational. We establish the existence of the desired traveling horseshoe
using induction by $m$. If $m=1$ (that is if $\ell_\ta$ is basic
non-rotational) then everything follows from Theorem~\ref{th-basnonrot}.
Suppose that the claim is proven for $m$ and prove it for $m+1$. If $\ell_\ta$
is rotational of depth $m+1$ then we consider the rotational set $A_\ta$
of period $k$. We see that $A_\ta\cap \ell_\ta=\0$ but $\{\ta, \ta'\}\subset
D_\iy(\ta)$ where $D_\iy(\ta)$ is the canonic traveling horseshoe associated to $A_\ta$
(and generated by $D$ where $D$ is the union of two appropriate intervals). The
critical leaf $\ell_{\ph(\ta)}$ is rotational of depth $m$, hence by
induction there exist two intervals $J, J'\subset \ucirc$ whose union $Q$
generates a traveling horseshoe $Q_\iy$ of period $l$ satisfying (together with
$\ell_{\ph(\ta)}$) conditions of Theorem~\ref{th-basnonrot}. Consider the
intervals $I=\ph^{-1}(J)$ and $I'=\ph^{-1}(J')$ and their union $\hDi=I\cup I'$.

Observe that $\ph$ is not one-to-one only on preimages of $0$. Hence $\ph$ is
one-to-one on preimages of the endpoints of $J, J'$. We may assume that
$I=[\al, \be]$ and $I'=[\al', \be']$, and it follows that $0, \frac 12\nin \hDi$.
Since $Q_\iy$ is of period $l$ and $A$ is of period $k$ then the appropriate
endpoints of $I, I'$ are of $\si$-period $kl$. Then $I, I'$ generate a general
horseshoe $\hDi_\iy$ and we want to prove that $\hDi_\iy$ together with
$\ell_\ta$ satisfy the conditions of Theorem~\ref{th-basnonrot}. First recall
that the union of 4 intervals $H(Q)$ constructed in the definition of a general
horseshoe does not contain $0$ or $\frac 12$. Hence every point of $H(\hDi)$
comes back into $D$ under $\si^k, \si^{2k}, \dots, \si^{lk}$ which implies that
$\hDi_\iy\subset D_\iy$. Moreover, $\ph:\hDi_\iy\to Q_\iy$ is a conjugacy since
$\ph$ only collapses holes of $D_\iy$ which eventually map onto $\ph$-preimage
of $0$ while on the other hand $0\nin Q_\iy$.

Since $Q_\iy$ is a traveling horseshoe of period $l$ then the convex hulls of
sets $\si(Q_\iy), \dots, \si^{l-1}(Q)$ are disjoint (except possibly for the
boundaries) from the convex hull of $Q_\iy$. The same holds for the convex
hulls of the sets $\hDi_\iy, \si^k(\hDi_\iy), \dots, \si^{k(l-1)}(\hDi_\iy)$. We
need to show that actually the convex hulls of sets $\si(\hDi_\iy), \dots,
\si^{kl-1}(\hDi_\iy)$ are disjoint from the convex hull of the set $\hDi_\iy$.
However it easily follows from the fact that $A$ is rotational and the
appropriate description of the dynamics on arcs complementary to $A$. Finally,
since $Q_\iy$ is a traveling horseshoe satisfying (together with the critical
leaf $\ell_{\ph(\ta)}$) conditions of Theorem~\ref{th-basnonrot} then the
properties of $\ph$ imply that so does the traveling horseshoe $\hDi_\iy$ and
the critical leaf $\ell_\ta$. By Theorem~\ref{th-basnonrot} we conclude that
there exists a lamination $\sim$ compatible with $\ell_\ta$.
\end{proof}

\bibliographystyle{plain}
\bibliography{c:/lex/references/refshort}

\end{document}